\documentclass[11pt]{article}

\usepackage{enumerate}

\usepackage{amsthm,amsmath,amssymb}
\usepackage{graphicx}
\usepackage[colorlinks=true,citecolor=black,linkcolor=black,urlcolor=blue]{hyperref}

\theoremstyle{plain}
\newtheorem{theorem}{Theorem}
\newtheorem{lemma}[theorem]{Lemma}

\newtheorem{proposition}[theorem]{Proposition}

{\itshape}{\rmfamily}

\theoremstyle{definition}

\newtheorem{example}[theorem]{Example}

\theoremstyle{remark}
\newtheorem{remark}[theorem]{Remark}

\title{\bf Uniform hypergraphs and dominating sets of graphs}

\author{Jaume Mart\'{\i}-Farr\'e\thanks{Partially supported by projects MINECO MTM2014-60127-P and Gen.Cat. DGR2014SGR1147.}
\qquad Merc\`e Mora \thanks{Partially supported by projects MINECO MTM2015-63791-R and Gen.Cat. DGR2014SGR46.}
\qquad Jos\'e Luis Ruiz\\
\small Departament de Matem\`atiques\\[-0.8ex]
\small Universitat Polit\`ecnica de Catalunya\\[-0.8ex]
\small Spain\\
\small\tt \{jaume.marti,merce.mora,jose.luis.ruiz\}@upc.edu}

\begin{document}


\maketitle

\begin{abstract}
A (simple) hypergraph is a family ${\cal H}$ of pairwise incomparable sets of a finite set $\Omega$. We say that a hypergraph ${\cal H}$ 
is a \emph{domination hypergraph} 
if there is at least a graph $G$ such that the collection of minimal dominating sets of $G$ is equal to ${\cal H}$. 
Given a hypergraph, we are interested in determining if it is a domination hypergraph and, 
if this is not the case, we want to find domination hypergraphs in some sense close to it; the {\em domination completions}. 
Here we will focus on the family of hypergraphs containing all the subsets with the same cardinality; 
the \emph{uniform hypergraphs} of maximum size. 
Specifically, we characterize those hypergraphs ${\cal H}$ in this family that are domination hypergraphs and, 
in any other case, we prove that the domination completions exist. Moreover,
we then demonstrate that  
the hypergraph ${\cal H}$ is uniquely determined by some of its domination completions, 
in the sense that ${\cal H}$ can be recovered from its minimal domination completions by using a suitable hypergraph operation.

\bigskip\noindent \textbf{Keywords:} graphs; hypergraphs; uniform hypergraphs; dominating sets.
\end{abstract}


\def\DD{\mathcal{D}}
\def\Dom{\mathop{\mathrm{Dom}}}
\def\DDom{\mathcal{D}om}


\section{Introduction}


A \emph{vertex dominating set} of a graph $G$
is a set of vertices $D$ such that every vertex of $G$ is either in $D$ or adjacent to some vertex of $D$ (see~\cite{HHS98}).
Domination in graphs is a widely researched branch of graph theory,
both from a theoretical and algorithmic point of view.
In part, it is due to its applications to several fields where graphs are used to model the relationships between a finite number of objects.
In this way, for instance,
some concepts from domination in graphs appear in problems involving finding sets of representatives,
as well as in facility location problems or in problems in monitoring communication, in electrical networks or in network routing.

The starting point of this work is a question concerning the design of networks on a finite set of nodes $\Omega$
whose dominating sets satisfy specific properties.
Thus, in this paper we focus our attention on the collection  $\DD(G)$
of all the \emph{inclusionwise minimal vertex dominating sets} of a graph $G$.
Specifically, we are looking for graphs $G$ whose collection of vertex dominating sets  $\DD(G)$
is equal or close
to a given collection $\{A_1, \dots, A_r\}$ of subsets of nodes $A_i\subseteq \Omega$. 
This problem is related to the closed neighborhood realization problem (see~\cite{BLS99}),
that was first proposed by S\'os under the name of star system problem (see~\cite{S78}),
and has been studied by different authors (see~\cite{AT93,BGZ08,FKLMT11,KLMN14}).

\emph{Hypergraphs} become the natural framework of this problem.
A (simple) \emph{hypergraph} ${\cal H}$ on a finite set
$\Omega$ is a collection of subsets of $\Omega$ none of which is a proper subset of another
(see~\cite{B89}).
The \emph{domination hypergraph of a graph} $G$
is the  collection $\DD(G)$
of all the inclusion-minimal vertex dominating sets of a graph $G$.
A hypergraph ${\cal H}$ is said to be a \emph{domination hypergraph} if ${\cal H}$ is the domination hypergraph of a graph;
that is, if ${\cal H}=\DD(G)$ for some graph $G$.

Since in general a hypergraph ${\cal H}$ is not the domination hypergraph of a graph,
a natural question that arises at this point is to determine domination hypergraphs close to ${\cal H}$, 
the \emph{domination completions} of ${\cal H}$. 
This paper deals with this issue.
Specifically, we focus our attention on this problem for the uniform hypergraphs 
${\cal H}={\cal U}_{r,\Omega}$ containing all the subsets with the same cardinality $r$ of a finite set $\Omega$.
The goal is to prove that the domination completions of the uniform hypergraph  ${\cal U}_{r,\Omega}$ exist. Moreover,
by taking into account a suitable partial order $\leqslant$, we will prove that the set of domination completions of ${\cal U}_{r,\Omega}$
is a partially ordered set, and that the uniform hypergraph ${\cal U}_{r,\Omega}$ is univocally determined by the minimal elements of this poset,
the \emph{minimal domination completions}. Namely
we will prove that there is a hypergraph operation $\sqcap$ that allows us to express the uniform hypergraph ${\cal U}_{r,\Omega}$ 
as a combination of its optimal completions
$\DD(G_1), \dots, \DD(G_s)$, that is, ${\cal U}_{r,\Omega}=\DD(G_1) \sqcap \cdots \sqcap \DD(G_s)$. 
We thus speak of a \emph{decomposition} of ${\cal U}_{r,\Omega}$.
In addition, we study the number of
completions appearing in the decomposition of ${\cal U}_{r,\Omega}$.

Summarizing, in this paper we present new results concerning with the completions and decompositions of hypergraphs 
into \emph{domination hypergraphs} (a previous version of this work was presented at the European Conference on Combinatorics, 
Graph Theory and Applications - EUROCOMB 2015~\cite{JMJL201501}).
Closest in spirit to our work is the paper~\cite{MdM2015}, in which the authors present 
some results on the completion and decompositions of hypergraphs into \emph{matroidal hypergraphs}. 
It is worth mentioning that even though our results are formally analogous to those in~\cite{MdM2015}, their proofs are quite different.

The paper is structured as follows. First, in Section~\ref{sec general} 
we recall the properties of vertex dominating sets of graphs that we will use throughout the paper (Subsection~\ref{domgrafs});
we present the basic definitions on hypergraphs and domination hypergraphs (Subsection~\ref{domhypgrafs});
and we characterize the uniform hypergraphs of maximum size that are domination hypergraphs (Subsection~\ref{subuniform}).
The main theoretical results of this paper are gathered in Section~\ref{compdecomp}. In this section we
introduce the poset of domination completions (Subsection~\ref{subcomp}),
and we present our results on domination completion and decomposition of the uniform hypergraphs
${\cal U}_{r,\Omega}$ (Theorem~\ref{thcompletacio}, Theorem~\ref{thdecom} and Proposition~\ref{recup2}). 
Finally, in Section~\ref{mincomp}, we describe the set of the minimal domination completions of some uniform hypergraphs ${\cal U}_{r,\Omega}$;
we present their graph realization; and, we discuss some issues on the corresponding domination decomposition.
Concretely, we analyze these questions for the uniform hypergraphs ${\cal U}_{r,\Omega}$ when $r=2$ (Subsection~\ref{section2n}), when $r=|\Omega|-1$
(Subsection~\ref{sectionn-1n}), and when $r$ arbitrary and $|\Omega|\leq 5$ (Subsection~\ref{section135}).


\section{Dominating sets of graphs. Domination hypergraphs}\label{sec general}


As mentioned in the introduction, the aim of this section is to present those general results on dominating 
sets that we will use throughout the paper. 


\subsection{Vertex dominating sets of graphs}\label{domgrafs}


A \emph{graph} $G$ is an ordered pair $(V(G),E(G))$ comprising a finite set $V(G)$
of \emph{vertices} together with a (possible empty) set $E(G)$ of \emph{edges} which
are two-element subsets of $V(G)$ (for general references on graph theory see~\cite{CL05, W01}). 
If $e=\{x,y\}\in E(G)$ is an edge of $G$, then $x$ and $y$ are 
said to be \emph{adjacent vertices}.
An \emph{isolated vertex} is a vertex of the graph that is not adjacent to any other vertices; 
that is, a vertex that does not belong to any edge of the graph.
Let us denote by $V_0(G)$ the set of all the isolated vertices of $G$.

A \emph{dominating set} for a graph $G=(V(G),E(G))$ is a subset $D$ of $V(G)$ such that every vertex not in $D$
is adjacent to at least one member of $D$.  Since any superset of a dominating set of $G$ 
is also a dominating set of $G$, the collection $D(G)$
of the dominating sets of a graph $G$ is a monotone increasing family of subsets of
$V(G)$. Therefore, $D(G)$ is uniquely determined by the family $\min\big(D(G)\big)$ of its inclusion-minimal
elements. Let us denote by $\DD(G)$ the family of the \emph{inclusion-minimal dominating sets} of the graph $G$.

Dominating sets of a graph are closely related to independent sets.
An \emph{independent set} of a graph $G$ is a set of vertices such that no two of them are adjacent.
It is clear that an independent set is also a dominating set if and only if it is an inclusion-maximal independent set  (see~\cite{CL05}).
Therefore, any inclusion-maximal independent set of a graph is necessarily also an inclusion-minimal dominating set.
The next lemma follows from this fact and from the definitions.

\begin{lemma}\label{lema.unio}
Let $G$ be a graph.
Then, $V(G) = \bigcup_{D\in \DD(G)} D$ and $V_0(G) =\bigcap_{D\in \DD(G)} D $.
\end{lemma}

Next, in Lemma~\ref{entorn.prop}, we recall the well-known relation between dominating sets and \emph{star systems}
(see~\cite{BGZ08,BLS99,KLMN14}).

The \emph{star system of a graph} $G=(V(G),E(G))$ is the multiset $N[G]$  of \emph{closed neighborhoods} of all the vertices of the graph;
that is, the multiset $N[G] = \{N[x] \,:\, x\in V(G)\}$
where $N[x]=\{x\}\cup \{y\in V(G) \,:\, \{x,y\}\in E(G)\}$. 
Let us denote by $\mathcal{N}[G]$ the inclusion-minimal elements of the star system; that is, 
$\mathcal{N}[G]=\min \big ( N [G] \big )$
is the family of the \emph{inclusion-minimal closed neighborhoods of the graph} $G$.
The relation between $\DD(G)$ and $\mathcal{N}[G]$ involves the \emph{transversal} or
\emph{blocker} of a family of subsets. Let ${\cal A}$ be a collection of subsets none of which is a proper subset of another.
The \emph{transversal} $tr({\cal A})$ of the family ${\cal A}$ consists of those inclusion-minimal subsets that have non-empty intersection with every
member of ${\cal A}$; that is, $tr({\cal A})=\min \{X \, : \, X\cap A\neq \emptyset$ for all $A\in {\cal A} \}$.

\begin{lemma}\label{entorn.prop}
Let $G$ be a graph. Then,  $\DD(G)=tr \big (\mathcal{N}[G]\big )$ and  $\mathcal{N}[G]=tr \big (\DD(G)\big )$.
\end{lemma}

\begin{proof}
From the definitions it is clear that a subset $D$ of vertices is a dominating set of
the graph $G$  if and only if $D\cap N[x]\not= \emptyset$ for every vertex $x\in V(G)$. Hence it follows that
$\DD(G)=tr \big (\mathcal{N}[G]\big )$.
The transversal map is involutive, that is, $tr(tr({\cal A}))={\cal A}$ (see~\cite{B89}). Therefore
we get that $\mathcal{N}[G]$ is, at once,  the transversal of the family $\DD(G)$.
\end{proof}

To conclude this subsection we recall two graph operations that we will use: the \emph{disjoint union} and the \emph{join} of graphs.

Let $G_1, \dots, G_r$ be $r\geq 2$ graphs with pairwise disjoint vertex sets $V(G_1),\dots, V(G_r)$.
The \emph{disjoint union} $G_1+\cdots +G_r$ of $G_1, \dots, G_r$
is the graph with $V(G_1)\cup \cdots \cup V(G_r)$ as set of vertices and $E(G_1)\cup \cdots \cup E(G_r)$ as set of edges;
while the \emph{join} $G_1\vee \cdots \vee G_r$ of $G_1, \dots, G_r$
is the graph with set of vertices $V(G_1)\cup \cdots \cup V(G_r)$ and set of edges
$E(G_1)\cup \cdots \cup E(G_r)\cup \{ \{x_{1},x_{2}\} \, :\, x_{1}\in V(G_{i_1}), x_{2}\in V(G_{i_2})$ and $i_1\neq i_2\}$.
The following lemma deals with the minimal dominating sets of these graphs. Its proof is a
straightforward consequence of the definitions.

\begin{lemma}\label{lema.operacions}
Let $G_1,\dots, G_r$ be $r\geq 2$ graphs with pairwise disjoint set of vertices. Then:
\begin{enumerate}
\item $\DD (G_1+\cdots +G_r)=\{ D_1 \cup \cdots \cup D_r : D_i\in \DD (G_i)\}$.
\item $\DD (G_1\vee \cdots \vee G_r)=\\
\DD (G_1) \cup \cdots \cup  \DD (G_r) 
\cup \{ \{x_{1},x_{2}\} \, : \, x_{k}\in V(G_{i_k})$, $N[x_{k}]\neq V(G_{i_k})$, $i_1\neq i_2\}$. 
\end{enumerate}
\end{lemma}


\subsection{Hypergraphs. Domination hypergraphs}\label{domhypgrafs}


Let $\Omega$ be a non-empty finite set. A \emph{(simple) hypergraph} on $\Omega$
is a non-empty collection ${\cal H}$ of non-empty different subsets of $\Omega$, none of which is a proper subset of another;
that is, if $A,A'\in {\cal H}$ and $A\subseteq A'$ then $A=A'$.
Hypergraphs are also
known as \emph{antichains},  \emph{Sperner systems}  or \emph{clutters}
(for general references on hypergraph theory see~\cite{B89,D95}).
In general, if ${\cal H}$ is a hypergraph on $\Omega$ then $\bigcup_{A\in {\cal H}} A\subseteq \Omega$.
We say that ${\cal H}$ is a \emph{hypergraph with ground set $\Omega$} whenever the equality $\Omega=\bigcup_{A\in {\cal H}} A$ holds.

There are several hypergraphs that can be associated to a graph. In this paper we are
interested in those hypergraphs defined by the dominating sets
of the graph. Namely, if $G$ is a graph with vertex set $V(G)$,
we consider the collection $\DD(G)$ of the inclusion-minimal dominating sets of the graph.
It is clear that $\DD(G)$ is a hypergraph on the finite set $V(G)$. Moreover, by Lemma~\ref{lema.unio},
$\DD(G)$  is a hypergraph with ground set $V(G)$.

The \emph{domination hypergraphs} are those hypergraphs that can be realized by the dominating sets of a graph; that is,
we will say that a hypergraph ${\cal H}$ on $\Omega$ is a \emph{domination hypergraph}
if there exists a graph $G$ such that ${\cal H}=\DD(G)$
(notice that then the set of vertices of $G$ is $V(G)=\bigcup_{A\in {\cal H}} A \subseteq \Omega$). If ${\cal H}=\DD(G)$, we say that the graph $G$ 
is a \emph{realization}  of the domination hypergraph ${\cal H}$.

\begin{remark}
Observe that there exist domination hypergraphs ${\cal H}$ with more than one graph realization.
For example, let us consider the hypergraph ${\cal H}=\{\{1,3\},\{1,4\},\\ \{2,3\},\{2,4\}\}$ on the finite set $\Omega=\{1,2,3,4\}$.
Then ${\cal H}= \DD(G)=\DD(G')=\DD(G'')$ where $G$, $G'$ and $G''$ are the graphs
with vertex sets $V(G)=V(G')=V(G'')=\Omega$ and edge sets
$E(G)=\{\{1,2\}, \{3,4\}\}$,
$E(G')=\{\{1,2\},\{2,3\},\{3,4\}\}$
and $E(G'')=\{\{1,2\},\{1,4\},\{3,4\}\}$.
\end{remark}

The following lemma provides a necessary condition for a hypergraph to be a domination hypergraph.

\begin{lemma}\label{cond.necesari}
Let ${\cal H}$ be a hypergraph with ground set $\Omega$. Assume that ${\cal H}$ is a domination hypergraph.
Then,  $|tr \big ({\cal H} \big )|\leq |\Omega|$.
\end{lemma}

\begin{proof}
Let $G$ be a graph with vertex set $V(G)=\Omega$ and such that ${\cal H}=\DD(G)$. Then, by applying Lemma~\ref{entorn.prop}
we get that $tr \big ({\cal H} \big )=tr \big (\DD(G)\big )=
\mathcal{N}[G]$. So, $|tr \big ({\cal H} \big )|= |\mathcal{N}[G]| \leq |V(G)|=|\Omega|$.
\end{proof}

From the above, not all hypergraphs are domination hypergraphs. Indeed, let
${\cal A}=\{A_1, \dots , A_r \}$ be a family of $r\geq |\Omega| +1$ 
non-empty different subsets of $\Omega$ with $A_i\not\subseteq A_j$ if $i\neq j$
(for instance, the family  ${\cal A}=\{ A \subseteq \Omega \, : \, |A|=2\}$ where $|\Omega|\geq 4$).
Since $tr(tr({\cal A}))={\cal A}$,
from Lemma~\ref{cond.necesari} it follows that the hypergraph ${\cal H}=tr({\cal A})$ is not a domination hypergraph.

Therefore, a natural question that arises at this point is to characterize
whenever a hypergraph ${\cal H}$ is a domination hypergraph. The following subsection
deals with this issue for a special family of hypergraphs.


\subsection{Uniform hypergraphs. Domination hypergraphs of the form ${\cal U}_{r,\Omega}$}\label{subuniform}


Let $\Omega$ be a finite set of size $|\Omega|=n$
and let  $1\leq r\leq n$.
We say that a hypergraph ${\cal H}$ on $\Omega$ is \emph{$r$-uniform} if $|A|=r$ for all $A\in {\cal H}$.
Let us denote by ${\cal U}_{r,\Omega}$ the $r$-uniform hypergraph on $\Omega$ whose elements are all the subsets of $\Omega$ of size $r$;
that is,  ${\cal U}_{r,\Omega}=\{A\subseteq\Omega \, : \, |A|=r\}$.

The following proposition provides a characterization
of the domination hypergraphs of
the form ${\cal U}_{r,\Omega}$, as well as the description
of their graph realizations.
This proposition was partially stated in~\cite{J97}.

Before stating the proposition, let us introduce some notation. The \emph{complete graph} with $n$ vertices
is denoted by $K_n$, whereas
the \emph{complete graph} with vertex set $\Omega$ will be denoted by $K_{\Omega}$, and
the \emph{empty graph} with vertex set $\Omega$ will be denoted by $\overline{K_{\Omega}}$.
Observe that if $|\Omega|=2m$, then
the graph $G$ obtained from the complete graph $K_{\Omega}$ by deleting the edges of a perfect matching
is the join graph of empty graphs on sets of size two; that is,
$G=\overline{K_{\Omega_1}}\vee \cdots \vee \overline{K_{\Omega_m}}$ where $\Omega= \Omega_1 \cup \cdots \cup \Omega_m$ and $|\Omega_i|=2$ for
$1\leq i \leq m$ (namely the vertices of the sets $\Omega_i$ are the endpoints of each one of the edges of the perfect matching).

\begin{proposition}\label{domuniform.caract}
Let $\Omega$ be a finite set of size $|\Omega|=n$. Let $1\leq r\leq n$. Then,
the hypergraph  ${\cal U}_{r,\Omega}$ is a domination hypergraph if and only if $r=1$, or $r=n$, or $r=2$ and  $n$ is even.
Moreover,  the following statements hold:
\begin{enumerate}
\item The complete graph $K_{\Omega}$ is the unique graph $G$ such that ${\cal U}_{1,\Omega}=\DD(G)$.
\item The empty graph $\overline{K_{\Omega}}$ is the unique graph $G$ such that ${\cal U}_{n,\Omega}=\DD(G)$.
\item If $n=2m$, then there are $(2m)!/(2^mm!)$ graphs $G$ such that  ${\cal U}_{2,\Omega}=\DD(G)$.
Namely, $G$ is any graph of the form
$G=\overline{K_{\Omega_1}}\vee \cdots \vee \overline{K_{\Omega_m}}$ where $\Omega= \Omega_1 \cup \cdots \cup \Omega_m$ and $|\Omega_i|=2$ for
$1\leq i \leq m$.
\end{enumerate}
\end{proposition}

\begin{proof}
Statements (1) and (2) are a straightforward consequence of the definitions. Let us prove the third statement.

Assume that $n=2m$ is even.
From the description of the minimal domination sets of the join graph it follows that
${\cal D}(\overline{K_{\Omega_1}}\vee \cdots \vee \overline{K_{\Omega_m}})={\cal U}_{2,\Omega}$
if $\Omega= \Omega_1 \cup \cdots \cup \Omega_m$ and $|\Omega_i|=2$ for
$1\leq i \leq m$. So the uniform hypergraph ${\cal U}_{2,\Omega}$ is a domination hypergraph and the graphs
of the form $G=\overline{K_{\Omega_1}}\vee \cdots \vee \overline{K_{\Omega_m}}$ are domination realizations of
${\cal U}_{2,\Omega}$. Conversely, let us prove that if $G$ is a graph such that $\DD(G)={\cal U}_{2,\Omega}$ then $G$ is obtained
from the complete graph $K_{\Omega}$ by deleting the edges of a perfect matching.
So, assume that $\DD(G)={\cal U}_{2,\Omega}$.
Then from Lemma~\ref{entorn.prop} it follows that
$\mathcal{N}[G]=tr({\cal U}_{2,\Omega})$. Since
$tr({\cal U}_{2,\Omega})={\cal U}_{2m-1,\Omega}$, hence
$\mathcal{N}[G]={\cal U}_{2m-1,\Omega}$, and
therefore  all the vertices of $G$ have degree $2m-2$. Consequently, the graph $G$ is obtained
from the complete graph $K_{\Omega}$ by deleting the edges of a perfect matching, as we wanted to prove.

From the above we conclude that if $n=2m$ is even, then
${\cal U}_{2,\Omega}=\DD(G)$ if and only if
$G$ is a graph obtained from $K_{\Omega}$ by deleting the edges of a perfect matching.
It is well known that
the number of perfect matchings in a complete graph
$K_{2m}$ is given by the double factorial $(2m-1)!!$, that is, $(2m)!/(2^mm!)$.
Hence, if $n=2m$ is even, then there are $(2m)!/(2^mm!)$ graphs $G$ such that $\DD(G)={\cal U}_{2,\Omega}$. This completes the proof of the third statement.

To finish the proof of the proposition we must demonstrate that if $3\leq r \leq n-1$, or if $r=2$ and $n$ is odd, then
${\cal U}_{r,\Omega}$ is not a domination hypergraph. Otherwise, assume that there exists a graph $G$ with vertex set $V(G)=\Omega$ and such that
$\DD(G)={\cal U}_{r,\Omega}$. Since  $tr({\cal U}_{r,\Omega})={\cal U}_{n-r+1,\Omega}$,
from Lemma~\ref{entorn.prop} we get that  $\mathcal{N}[G]={\cal U}_{n-r+1,\Omega}$.
On one hand, the size of $\mathcal{N}[G]$ is at most $n$ because $V(G)=\Omega$. On the other hand,
${\cal U}_{n-r+1,\Omega}$ has size $n\choose n-r+1$. Therefore
${n\choose n-r+1} \leq n$, and thus $r=2$.
At this point we have that $G$ is a graph of order $n$ with  $\mathcal{N}[G]={\cal U}_{n-r+1,\Omega}={\cal U}_{n-1,\Omega}$.
So, $G$ is a $(n-2)$-regular graph of order $n$,
which  is not possible if $n$ is odd.
This completes the proof of the proposition.
\end{proof}


\section{Domination completions and decompositions of the unifom hypergraphs ${\cal U}_{r,\Omega}$}\label{compdecomp}


Not all hypergraphs are domination hypergraphs. Therefore,
given a hypergraph ${\cal H}$ a natural question is to study domination hypergraphs ``close" to ${\cal H}$.
Here, we consider this problem whenever ${\cal H}$ is the \emph{uniform} hypergraph ${\cal U}_{r,\Omega}$.
Our goal is to introduce the poset of domination completions of ${\cal U}_{r,\Omega}$ (Subsection~\ref{subcomp}),
and to prove that the minimal elements of this poset provide a decomposition of ${\cal U}_{r,\Omega}$ into domination hypergraphs
(Subsection~\ref{subdecomp}).


\subsection{Poset of domination completions of ${\cal U}_{r,\Omega}$}\label{subcomp}


Let $\Omega$ be a finite set and let us consider the $r$-uniform hypergraph
${\cal U}_{r,\Omega}$. Observe that if ${\cal H}$ is a hypergraph on $\Omega$, the elements of ${\cal H}$ are pairwise non-comparable sets,
and hence it follows that ${\cal U}_{r,\Omega} \subseteq {\cal H}$
if and only if ${\cal U}_{r,\Omega} = {\cal H}$. Therefore, if ${\cal U}_{r,\Omega}$ is not a domination hypergraph,
then there does not exist a graph $G$ with vertex set $V(G)=\Omega$
such that ${\cal U}_{r,\Omega} \subseteq \DD(G)$.
Thus, a question that arises at this point is to determine domination hypergraphs ${\cal H}$ \emph{close} to the hypergraph ${\cal U}_{r,\Omega}$.
This subsection deals with the existence of such \emph{domination completions} ${\cal H}$ of ${\cal U}_{r,\Omega}$.

A crucial point when looking for the domination completions ${\cal H}$ of ${\cal U}_{r,\Omega}$
is to take into account all the dominating sets of a graph $G$ instead of considering
only the inclusion-minimal dominating sets of $G$; that is,
taking into account the family $D(G)$
instead of the family $\DD(G)$. Specifically, in order to seek the domination hypergraphs close to ${\cal U}_{r,\Omega}$
we introduce a suitable partial order $\leqslant$ on the set of hypergraphs that involves the \emph{monotone
increasing family of subsets} ${\cal H}^+$ associated to a hypergraph ${\cal H}$.

Let $\Omega$ be a finite set. Let ${\cal H}$ be a hypergraph on $\Omega$.
Then we define ${\cal H}^+$ as the family whose elements are the subsets $A\subseteq \Omega$ such that there exists
$A_0\in {\cal H}$ with  $A_0\subseteq A$.
Observe that ${\cal H}^+$ is a monotone increasing family of subsets of $\Omega$ whose inclusion-minimal
elements are the subsets of ${\cal H}$; that is,  ${\cal H}=\min\big( {\cal H}^+ \big)$. Therefore, the hypergraph  ${\cal H}$ is
uniquely determined by the monotone increasing family ${\cal H}^+$. For instance, if $G$ is a graph then $\DD(G)$ is
a hypergraph on $V(G)$ whose associated monotone increasing family of subsets is $\DD(G)^+=D(G)$,
and so $D(G)$ is uniquely determined by $\DD(G)$.

To compare two hypergraphs ${\cal H}_1,{\cal H}_2$ on $\Omega$, we use their associated monotone
increasing families of subsets ${\cal H}_1^+,{\cal H}^+_2$.
It is clear that if ${\cal H}_1 \subseteq {\cal H}_2$,  then ${\cal H}_1^+\subseteq {\cal H}_2^+$.
However, the converse is not true; that is, there exist hypergraphs with ${\cal H}_1\not \subseteq {\cal H}_2$
and ${\cal H}_1^+\subseteq {\cal H}_2^+$ (for instance the hypergraphs ${\cal H}_1=\{ \{1,2\},\{1,3\},\{2,3\}\}$ and ${\cal H}_2=\{ \{1,2\},\{ 3 \} \}$).
This fact leads us to consider the binary relation $\leqslant$ defined as follows:
if ${\cal H}_1$ and ${\cal H}_2$ are two hypergraphs on the finite set $\Omega$,
then we say that
$${\cal H}_1\leqslant {\cal H}_2 \textrm{ if and only if } {\cal H}_1^+\subseteq {\cal H}_2^+.$$

\begin{lemma}
\label{lem.ordre}
Let $\Omega$ be a finite set. The following statements hold:
\begin{enumerate}
\item If ${\cal H}_1,{\cal H}_2$ are two hypergraphs on $\Omega$
then, ${\cal H}_1\leqslant {\cal H}_2$ if and only if for all $A_1\in {\cal H}_1$ there exists
$A_2\in {\cal H}_2$ such that $A_2\subseteq A_1$.
\item The binary relation $\leqslant$ is a partial order on the set
of hypergraphs on $\Omega$.
\end{enumerate}
\end{lemma}

\begin{proof}
The proofs of the statements are a straightforward consequence of the definition of the family ${\cal H}^+$
and of the fact that  ${\cal H}=\min ({\cal H}^+) $.
\end{proof}

Now, by using the partial order $\leqslant$, we define the \emph{domination completions} of the hypergraph ${\cal U}_{r,\Omega}$ as 
any domination hypergraph ${\cal H}$ with ground set $\Omega$ such that ${\cal U}_{r,\Omega}\leqslant {\cal H}$.
We denote by  $\Dom(r,\Omega)$ the set whose elements are the domination completions of the hypergraph ${\cal U}_{r,\Omega}$;
that  is,
$$ \Dom (r,\Omega) =\big \{{\cal H} \, : \,
{\cal H} \textrm{ is a domination hypergraph with ground set $\Omega$ and } {\cal U}_{r,\Omega} \leqslant {\cal H}\big \}.$$

\begin{proposition}\label{nobuit}
Let $\Omega$ be a finite set. Then,  the set $\Dom(r,\Omega)$  is non-empty.
\end{proposition}

\begin{proof}
The uniform hypergraph ${\cal U}_{1,\Omega}$ is a domination hypergraph,
and it is clear that ${\cal U}_{r,\Omega} \leqslant {\cal U}_{1,\Omega}$.
So ${\cal U}_{1,\Omega} \in \Dom (r,\Omega)$.
\end{proof}

The partial order $\leqslant$ induces a poset structure in the set $\Dom(r,\Omega)$
of the domination completions of the hypergraph ${\cal U}_{r,\Omega}$.
The minimal elements of the
poset  $\big(\Dom(r,\Omega),\leqslant\big)$ are
the \emph{optimal} domination completions of ${\cal U}_{r,\Omega}$. This subsection is completed by showing that if  ${\cal U}_{r,\Omega}$
is not a domination hypergraph, then ${\cal U}_{r,\Omega}$ has at least two different optimal domination completions (Theorem~\ref{thcompletacio}).
The following technical lemma is a key point in order to prove this result.

\begin{lemma}
\label{evitacio}
Let $\Omega$ be a finite set.
Let ${\cal H}$ be a hypergraph on $\Omega$ such that ${\cal U}_{r,\Omega}\leqslant {\cal H}$ and ${\cal U}_{r,\Omega}\neq {\cal H}$.
Then, there exists a domination hypergraph ${\cal H}_0\in \Dom(r,\Omega)$ such that ${\cal H}\not\leqslant {\cal H}_0$.
\end{lemma}

\begin{proof}
First notice that for $r=1$ the hypotheses of the lemma do not hold because there is no hypergraph ${\cal H}$ on $\Omega$ 
different from  ${\cal U}_{1,\Omega}$ such that ${\cal U}_{1,\Omega}\leqslant {\cal H}$.
Moreover, observe that if $r=n$, then ${\cal U}_{r,\Omega}$ is a domination hypergraph, and so 
the hypergraph ${\cal H}_0={\cal U}_{n,\Omega}$ fulfills the required conditions.
Therefore, from now on we may assume that $2\le r\le n-1$.

Let $2\le r\le n-1$.
By assumption, ${\cal U}_{r,\Omega}\leqslant {\cal H}$ and ${\cal U}_{r,\Omega}\neq {\cal H}$.
Thus, since $\leqslant$ is a partial order, it follows that ${\cal H} \not\leqslant {\cal U}_{r,\Omega}$.
Therefore, there exists $A_0\in{\cal H}$ such that $A\not\subseteq A_{0}$ for all $A\subseteq \Omega$ with $|A|=r$;
that is, there exists $A_0\in{\cal H}$ with $|A_0|=t<r$. Without loss of generality we may assume that
$\Omega=\{ w_1, \dots, w_t,w_{t+1}, \dots , w_r, \dots, w_n\}$ and that $A_0=\{ w_1, \dots, w_t\}$.
Set $\Omega_1=\{ w_1, \dots, w_t,w_{t+1}, \dots ,w_r\}$ and set $\Omega_2=\Omega \setminus \Omega_1$.

At this point let us consider the domination hypergraph ${\cal H}_0=\DD(G_0)$ 
where $G_0$ is the join
graph $G_0=\overline{K_{\Omega_1}}\vee K_{\Omega_2}$. So, from Lemma~\ref{lema.operacions}
we get that ${\cal H}_0=\{\Omega_1\}\cup \{\{w\} \, : \, w\in \Omega_2\}$.
The proof will be completed by showing that ${\cal U}_{r,\Omega} \leqslant {\cal H}_0$ and that ${\cal H}\not\leqslant {\cal H}_0$.

In order to prove the inequality ${\cal U}_{r,\Omega} \leqslant {\cal H}_0$ we must demonstrate that for all $A\in {\cal U}_{r,\Omega}$
there exists $A'\in {\cal H}_0$ such that $A'\subseteq A$. So let $A\subseteq \Omega$ with $|A|=r$. If $A=\Omega_1$ then
set $A'=\Omega_1\in {\cal H}_0$; whereas if $A\neq \Omega_1$, then there exists $w\in \Omega_2$ such that $w\in A$ and, so,
we can consider $A'=\{w\}\in {\cal H}_0$.

To finish we must demonstrate that ${\cal H}\not\leqslant {\cal H}_0$. On the contrary,
let us assume that ${\cal H}\leqslant {\cal H}_0$. Then, since $A_0\in {\cal H}$, there exists $A\in{\cal H}_0$ such that $A\subseteq A_0$.
So either $\Omega_1\subseteq A_0$ or there exists $w\in \Omega_2$ such that $\{w\}\subseteq A_0$. In any case a contradiction is obtained because,
by construction, $A_0\varsubsetneq \Omega_1$ and $A_0\cap \Omega_2=\emptyset$.
This completes the proof of the lemma.
\end{proof}

\begin{theorem}\label{thcompletacio}
Let $\Omega$ be a finite set. Then
the non-empty poset $\big(\Dom(r,\Omega),\leqslant\big)$ of the domination completions of the hypergraph ${\cal U}_{r,\Omega}$
has a unique minimal element if and only if
the hypergraph ${\cal U}_{r,\Omega}$ is a domination hypergraph.
\end{theorem}

\begin{proof}
It is clear that if ${\cal U}_{r,\Omega}$ is a domination hypergraph, then the poset $\big(\Dom(r,\Omega), \leqslant \big)$
has a unique minimal element, namely $\min\big(\Dom(r,\Omega), \leqslant \big)=\{{\cal U}_{r,\Omega}\}$.
Therefore, we must only prove that if the poset $\big(\Dom(r,\Omega),\leqslant\big)$ has
a unique minimal element, then ${\cal U}_{r,\Omega}$ is a domination hypergraph.
So let us assume that $\min\big(\Dom(r,\Omega), \leqslant\big) =\{{\cal H}\}$. On one hand, ${\cal U}_{r,\Omega}\leqslant {\cal H}$
because ${\cal H}\in\Dom(r,\Omega)$. On the other hand, since ${\cal H}$ is the unique minimal element of the poset
$\big(\Dom(r,\Omega),\leqslant\big)$,
we get that  ${\cal H}\leqslant {\cal H}_0$ for all domination hypergraph ${\cal H}_0\in \Dom(r,\Omega)$. Therefore, from Lemma~\ref{evitacio} we conclude that
${\cal U}_{r,\Omega}={\cal H}$. In particular,  ${\cal U}_{r,\Omega}$ is a domination hypergraph,
as we wanted to prove.
\end{proof}

In the following example we present the description of the domination completions of the $2$-uniform hypergraph
${\cal U}_{2,\Omega}$ where $\Omega=\{1,2,3\}$
(in Subsection~\ref{section2n} we study  the general case ${\cal U}_{2,\Omega}$ where $\Omega$ is a finite set of odd size).

\begin{example}\label{exemple23.2}
Let $\Omega=\{1,2,3\}$ and let us consider the $2$-uniform hypergraph ${\cal U}_{2,\Omega}$;
that is,  ${\cal U}_{2,\Omega}=\{\{1,2\},\{1,3\},\{2,3\}\}$.
From Proposition~\ref{domuniform.caract} we know that ${\cal U}_{2,\Omega}$ is not a domination hypergraph. Therefore, by applying
Theorem~\ref{thcompletacio} we conclude that
the non-empty poset $\big(\Dom(2,\Omega),\leqslant\big)$ 
has at least two minimal elements. Let us compute these minimal elements. Let $G$ be a graph with vertex set $V(G)=\{1,2,3\}$.
It is clear that, if $E(G)=\emptyset$ then  $\DD(G)=\big \{\{1,2,3\}\big \}$; whereas 
if $E(G)=\big \{\{a,b\}\big \}$ then $\DD(G)=\big \{\{a,c\},\{b,c\}\big\}$; while
if $E(G)=\big \{\{a,b\},\{a,c\}\big\}$ then $\DD(G)=\big \{\{a\},\{b,c\}\big \}$;
but if $|E(G)|=3$ then $\DD(G)=\big \{\{1\},\{2\},\{3\}\big \}$.
Therefore we conclude that ${\cal U}_{2,\Omega}\leqslant \DD(G)$ if and only if either $|E(G)|=2$ or $|E(G)|=3$. Thus, 
the hypergraph ${\cal U}_{2,\Omega}$ has four domination completions,
namely, the three domination hypergraphs defined by the graphs of size 2 and the domination
hypergraph
defined by the graph of size 3; that is, $\Dom (2,\Omega)=\big \{{\cal H}_{1},{\cal H}_{2},{\cal H}_{3},{\cal U}_{1,\Omega}\}$
where ${\cal H}_{i}=\{\{i\},\{j,k\}\}$ being $\{i,j,k\}=\{1,2,3\}$.
Observe that ${\cal H}_{0}\leqslant {\cal U}_{1,\Omega}$ and that
${\cal H}_{i}\not\leqslant{\cal H}_{j}$ if $i \neq j$.
Therefore, the poset $\big(\Dom(2,\Omega),\leqslant\big)$ has three minimal elements ${\cal H}_{1}$, ${\cal H}_{2}$ and ${\cal H}_{3}$.
So ${\cal H}_{1}$, ${\cal H}_{2}$, ${\cal H}_{3}$ are the optimal domination completions of ${\cal U}_{2,\Omega}$.
\end{example}


\subsection{Domination decompositions of ${\cal U}_{r,\Omega}$}\label{subdecomp}


Let $\Omega$ be a finite set.
We say that a domination hypergraph ${\cal H}$ is a \emph{minimal domination completion} of the $r$-uniform hypergraph ${\cal U}_{r,\Omega}$
if ${\cal H}$ is a  minimal element of the poset $\big(\Dom(r,\Omega), \leqslant \big )$.
Let us denote by $\DDom (r,\Omega)$ the set whose elements are the minimal domination completions ${\cal H}$ of ${\cal U}_{r,\Omega}$; that is,
$$ \DDom (r,\Omega)=\min\big(\Dom(r,\Omega),\leqslant \big ).$$

We have seen in Theorem~\ref{thcompletacio} that $\DDom(r,\Omega)$ has cardinality one if and only if
${\cal U}_{r,\Omega}$ is a domination hypergraph. The following theorem deals with the case of cardinality greater
than one, and shows that we can recover uniquely the hypergraph from the elements of $\DDom(r,\Omega)$.
Before stating our result we must introduce the hypergraph operation $\sqcap$.

Let $\Omega$ be a finite set and let ${\cal H}_1, \dots ,{\cal H}_{\ell}$ be $\ell$ hypergraphs on $\Omega$.
Then we define the hypergraph ${\cal H}_1 \sqcap \dots \sqcap {\cal H}_{\ell}$ as:
\begin{align*}
{\cal H}_1 \sqcap \dots \sqcap {\cal H}_{\ell}
&=\min\big ( {\cal H}_1^+ \cap \dots \cap {\cal H}_{\ell}^+ \big ).
\end{align*}
Observe that, from the definition of ${\cal H}^+$ it is not hard to prove that
\begin{align*}
{\cal H}_1 \sqcap \dots \sqcap {\cal H}_{\ell}
=\min\big \{ A_1 \cup \cdots \cup A_{\ell}\, : \, A_i\in {\cal H}_i\text{ for } 1\le i\le {\ell}\big \}.
\end{align*}

\begin{theorem}
\label{thdecom}
Let $\Omega$ be a finite set.
Let  $\DDom (r,\Omega)=\{{\cal H}_1, \dots, {\cal H}_{s}\}$ be the set of the minimal domination completions of ${\cal U}_{r,\Omega}$.
Then, $${\cal U}_{r,\Omega}={\cal H}_1 \sqcap \dots \sqcap {\cal H}_s.$$
\end{theorem}

\begin{proof}
Let us denote ${\cal H}_0= {\cal H}_1 \sqcap \dots \sqcap {\cal H}_s
=\min\big\{\,A_1\cup\cdots\cup A_s\,:\, A_i\in {\cal H}_i\text{ for } 1\le i\le s\big \}$.

First let us show that ${\cal U}_{r,\Omega}\leqslant {\cal H}_0$, that is,
we must prove that if $A\in {\cal U}_{r,\Omega}$ then there exists $A_0\in {\cal H}_0$ such that $A_0\subseteq A$.
So, let $A\in {\cal U}_{r,\Omega}$. Let $1\leq i \leq s$.
Since ${\cal U}_{r,\Omega}\leqslant {\cal H}_i$, for all $A\in {\cal U}_{r,\Omega},$ there exists
$A_i\in {\cal H}_i$ with $A_i\subseteq A$.
Therefore we get that $A_1 \cup \cdots \cup A_s \subseteq A$.
So there exists $A_0\in {\cal H}_0$ such that $A_0\subseteq A$.

Next we are going to prove that ${\cal U}_{r,\Omega}={\cal H}_0$.
Observe that if ${\cal U}_{r,\Omega}\neq {\cal H}_0$ then,
by applying Lemma~\ref{evitacio} we get that
there exists a domination hypergraph ${\cal H}'_0\in\Dom(r,\Omega)$ such that
${\cal H}_0\not\leqslant {\cal H}'_0$.
The proof will be completed by showing that  this leads us to a contradiction.
On one hand, since ${\cal H}'_0\in \Dom(r,\Omega)$
and $\DDom (r,\Omega)=\{{\cal H}_1, \dots, {\cal H}_s\}$,
we conclude that there exists $i_0\in \{1, \dots ,s\}$ such that ${\cal H}_{i_0}\leqslant {\cal H}'_0$.
On the other hand, from the definition of ${\cal H}_0$ and by applying Lemma~\ref{lem.ordre} it is easy to check
that the inequality ${\cal H}_0\leqslant {\cal H}_{i_0}$ holds. Therefore we conclude that ${\cal H}_0\leqslant  {\cal H}'_0$
because the binary relation $\leqslant$ is a partial order. Hence a contradiction is achieved, as we wanted to prove.
\end{proof}

The previous theorem leads us to the following definition.
Let $\Omega$ be a finite set. We say that a
family $\{{\cal H}_1,\ldots,{\cal H}_t\}$  of $t\geq1$ distinct domination hypergraphs
with ground set $\Omega$ is a \emph{$t$-decomposition} of the $r$-uniform hypergraph ${\cal U}_{r,\Omega}$
if ${\cal U}_{r,\Omega}={\cal H}_1 \sqcap \dots \sqcap {\cal H}_t$.
Let us denote $\mathfrak{D}(r,\Omega)=\min \{t \, : \, \textrm{there exists a $t$-decomposition of ${\cal U}_{r,\Omega}$}\}$.
It is clear that $\mathfrak{D}(r,\Omega)=1$
if and only if the $r$-uniform hypergraph ${\cal U}_{r,\Omega}$ is a domination hypergraph.

From Theorem~\ref{thdecom} we get that the domination hypergraphs in $\DDom(r,\Omega)$ provide a decomposition of ${\cal U}_{r,\Omega}$,
and therefore if the $r$-uniform hypergraph ${\cal U}_{r,\Omega}$ has $s$ minimal domination
completions, then $\mathfrak{D} (r,\Omega)\leq s$. The next proposition states that, in fact, to compute $\mathfrak{D} (r,\Omega)$ it is enough
to consider only those decompositions consisting of
minimal domination completions of  ${\cal U}_{r,\Omega}$.

\begin{proposition}
\label{recup2}
Let $\Omega$ be a finite set. Let $\mathfrak{D}(r,\Omega)=\delta$. Then
there exist $\delta$ minimal domination
completions ${\cal H}_{1}, \dots, {\cal H}_{\delta} \in \DDom (r,\Omega)$ of the $r$-uniform hypergraph ${\cal U}_{r,\Omega}$
such that $\{{\cal H}_{1}, \dots, {\cal H}_{\delta}\}$ is a $\delta$-decomposition of ${\cal U}_{r,\Omega}$.
\end{proposition}

\begin{proof}
To prove the proposition it is enough to show that any decomposition of ${\cal U}_{r,\Omega}$ can be transformed into a decomposition consisting
of minimal domination completions of ${\cal U}_{r,\Omega}$; that is,
we must demonstrate that if $\{{\cal H}'_1,\dots,{\cal H}'_t\}$ is a $t$-decomposition of ${\cal U}_{r,\Omega}$, then
there exist $\ell$ distinct hypergraphs ${\cal H}_{i_1}, \dots, {\cal H}_{i_{\ell}} \in \DDom (r,\Omega)$ (with $\ell\leq t$)
such that $\{{\cal H}_{i_1}, \dots, {\cal H}_{i_{\ell}}\}$ is an $\ell$-decomposition of ${\cal U}_{r,\Omega}$.

So, assume that $\{{\cal H}'_1,\dots,{\cal H}'_t\}$ is a $t$-decomposition of ${\cal U}_{r,\Omega}$, and
let $\{{\cal H}_1, \dots, {\cal H}_s\}=\DDom(r,\Omega)=\min \big (\Dom(r,\Omega), \leqslant \big )$.

It is clear that ${\cal H}_1' \sqcap \dots \sqcap {\cal H}_t '\leqslant {\cal H}'_k$ for all $k\in\{1,\dots, t\}$.
Therefore, ${\cal U}_{r,\Omega}\leqslant {\cal H}'_k$, and so ${\cal H}'_k\in \Dom(r,\Omega)$.
Since ${\cal H}_1, \dots, {\cal H}_s$ are the minimal elements of $(\Dom(r,\Omega), \leqslant \big )$,
for all $k\in \{1, \dots ,t\}$ there exists $\alpha_k\in \{1, \dots ,s\}$ such that ${\cal H}_{\alpha_k}\leqslant {\cal H}'_k$. Let
$\{{\cal H}_{i_1}, \dots, {\cal H}_{i_{\ell}}\}=\{{\cal H}_{\alpha_1}, \dots, {\cal H}_{\alpha_t}\}$, where ${\cal H}_{i_1}, \dots, {\cal H}_{i_{\ell}}$
are different (observe that $\ell\leq t$). 
On one hand we have that
${\cal H}_{i,1} \sqcap \dots \sqcap {\cal H}_{i_{\ell}}=
{\cal H}_{\alpha_1} \sqcap \cdots \sqcap {\cal H}_{\alpha_t}
\leqslant 
{\cal H}_1'\sqcap \dots \sqcap {\cal H}_{t}'= {\cal U}_{r,\Omega}$.
On the other hand, ${\cal U}_{r,\Omega}\leqslant
{\cal H}_{i_1} \sqcap \cdots \sqcap {\cal H}_{i_{\ell}}$
because ${\cal U}_{r,\Omega}\leqslant {\cal H}_{i_k}$ for all $k$.
Since $\leqslant$ is a partial order, we conclude that the equality
${\cal U}_{r,\Omega}={\cal H}_{i_1} \sqcap \cdots \sqcap {\cal H}_{i_{\ell}}$ holds; that is,
$\{{\cal H}_{i_1}, \dots, {\cal H}_{i_{\ell}}\}$ is an $\ell$-decomposition of ${\cal U}_{r,\Omega}$.
\end{proof}

To conclude this subsection let us show an example of decomposition, namely we compute the domination related parameter
$\mathfrak{D}(2,\Omega)$ where $\Omega=\{1,2,3\}$
(in Subsection~\ref{section2n} we study
the general case $\mathfrak{D}(2,\Omega)$ where $\Omega$ is a finite set of odd size).

\begin{example}\label{uniform23}
Let us consider the $2$-uniform hypergraph ${\cal U}_{2,\Omega}$
with ground set $\Omega=\{1,2,3\}$.
From Example~\ref{exemple23.2} we get that
${\cal U}_{2,\Omega}$ has three minimal domination completions, namely,
$\DDom(2,\Omega)=\{{\cal H}_{1},{\cal H}_{2}, {\cal H}_{3}\}$, where ${\cal H}_{i}=\{\{i\},\{j,k\}\}$ being $\{i,j,k\}=\{1,2,3\}$.
Therefore, by applying Theorem~\ref{thdecom} we get that ${\cal U}_{2,\Omega}={\cal H}_{1} \sqcap {\cal H}_{2} \sqcap {\cal H}_{3}$.
So the minimal domination completions of the non-domination hypergraph ${\cal U}_{2,\Omega}$ provides the domination decomposition
${\cal U}_{2,\Omega}=\allowbreak \DD(G_{1}) \sqcap \DD(G_{2}) \sqcap \DD(G_{3})$ where $G_{i}$ is the graph with vertex set $V(G_{i})=\{1,2,3\}$
and edge set $E(G_{i})=\{\{i,j\},\{i,k\}\}$ being $\{i,j,k\}=\{1,2,3\}$.
Thus we have that $2\leq \mathfrak{D}(2,\Omega) \leq 3$. However observe that
in this case, if $i_1 \neq i_2$, then ${\cal H}_{i_1}\sqcap  {\cal H}_{i_2}=
\big \{ \{i_1\},\{i_2,i_3\} \big \} \sqcap \big \{ \{i_2\},\{i_1,i_3\} \big \} =
\min \big \{ \{i_1,i_2\},\allowbreak  \{i_1,i_3\},\{i_2,i_3\},\{i_1,i_2,i_3\} \big \} =\big \{ \{i_1,i_2\},\{i_1,i_3\},\{i_2,i_3\} \big \}=
{\cal U}_{2,\Omega}$. Therefore we conclude that $\mathfrak{D}(2,\Omega) =2$.
\end{example}


\section{Determining the minimal domination completions of some uniform hypergraphs ${\cal U}_{r,\Omega}$}\label{mincomp}


Let $\Omega$ be a finite set of size $n$.
The study of the uniform hypergraphs of maximum size
will be completed with the computation
of the minimal domination completions of ${\cal U}_{r,\Omega}$ either when  $3\leq r\leq n-1$
or when $r=2$ and $n$ is odd. The description of the set $\DDom(r,\Omega)$ of all the minimal domination completions of ${\cal U}_{r,\Omega}$
is a problem which is far from being solved.
However, here we present the description in three cases.
Namely, the case $r=2$ and $n$ odd (Subsection~\ref{section2n}), the case $r=n-1$
(Subsection~\ref{sectionn-1n}), and the case $r$ arbitrary and $n\leq 5$ (Subsection~\ref{section135}).


\subsection{Minimal domination completions of ${\cal U}_{2,\Omega}$}\label{section2n}


From Proposition~\ref{domuniform.caract} and Theorem~\ref{thcompletacio} we get that $\DDom (2,\Omega)=\{{\cal U}_{2,\Omega}\}$ if and only if 
$\Omega$ has even size. In this subsection we determine the set $\DDom (2,\Omega)$ of the minimal domination completions 
of the uniform hypergraph ${\cal U}_{2,\Omega}$ whenever the finite $\Omega$ has odd size.

\begin{lemma}\label{lemacon}
Let $\Omega'$ be a non-empty subset of a finite set $\Omega$. 
Let ${\cal H}$ be a hypergraph on $\Omega$ and let ${\cal H}[\Omega']= \{ A\in \mathcal{H} : A\subseteq  \Omega' \}$.
Let $G'$ be a graph with vertex set $V(G')\subseteq \Omega'$, and suppose that ${\cal H}[\Omega']\neq \emptyset$. Then,
$ \mathcal{H}[\Omega'] \le \DD (G')$ if and only if $\mathcal{H}\le \DD (G'\vee K_{\Omega\setminus \Omega' }).$
\end{lemma}

\begin{proof}
First assuming that $\mathcal{H}[\Omega'] \le \DD (G')$
we are going to prove that $\mathcal{H}\le \DD (G'\vee K_{\{ \omega \} })$.
Recall that by Lemma~\ref{lema.operacions} we get that 
$\DD (G'\vee K_{\Omega\setminus \Omega' })=\DD(G')\cup \{ \{w\} : w\in \Omega\setminus \Omega'\}$.
Therefore, by applying Lemma~\ref{lem.ordre}, we must demonstrate that
if $A\in  \mathcal{H}$ then 
there exists $D\in \DD(G')\cup \{ \{w\} : w\in \Omega\setminus \Omega'\}$
such that $D\subseteq A$.
Let $A\in \mathcal{H}$.
If $A\not \subseteq \Omega'$, then there is $\omega \in A \cap (\Omega \setminus \Omega')$
and so we can set $D=\{\omega\}$.
Now assume that $A\subseteq \Omega'$. Then $A\in \mathcal{H}[\Omega']\le \DD (G')$, and hence
there exists $D'\in \DD (G')$ such that $D'\subseteq A$. Thus, in such a case, we can consider $D=D'$.

Now suppose that $\mathcal{H}\le \DD (G'\vee K_{ \Omega\setminus \Omega'})$. We want to prove that 
$\mathcal{H}[\Omega'] \le \DD (G')$; that is, we must demonstrate that if $A'\in \mathcal{H}[\Omega']$
then there exists $D'\in \DD(G')$ such that $D'\subseteq A'$.
Let $A'\in \mathcal{H}[\Omega']$.
Since $A'\in \mathcal{H}[\Omega']\subseteq {\cal H}$ and  ${\cal H}\leqslant \DD (G'\vee K_{\Omega \setminus \Omega' })$,
there exists $D\in \DD (G'\vee K_{\Omega \setminus \Omega'  })$ such that
$D\subseteq A'$. But $A'\subseteq \Omega' $ and, by Lemma~\ref{lema.operacions},
$\DD (G'\vee K_{\Omega\setminus \Omega' })=\DD(G')\cup \{ \{w\} : w\in \Omega\setminus \Omega'\}$.
Therefore $D\in \DD(G')$. Now the proof is completed by setting $D'=D$.
\end{proof}

\begin{theorem}\label{unidos}
Let $\Omega$ be a finite set of size $|\Omega|=n$.
Assume that $n$ is odd. Then the following statements hold:
\begin{enumerate}
\item For all $w\in \Omega$, the hypergraph ${\cal H}_{\omega}=\big \{\{\omega\}\big \}\cup {\cal U}_{2,\Omega\setminus\{\omega\}}$
is a domination hypergraph. Moreover, if $G$ is a graph with vertex set $\Omega$, then
$\DD(G)={\cal H}_{\omega}$ if and only if
$G=K_{ \{\omega\} }\vee G'$ where $G'$ is a realization of the domination hypergraph ${\cal U}_{2,\Omega\setminus\{\omega\}}$.
\item The uniform hypergraph
${\cal U}_{2,\Omega}$ has $n$ minimal domination completions. Namely,
if $\Omega=\{ \omega_1,\dots ,\omega_{n}\}$ then
$\DDom (2,\Omega)=\{\mathcal{H}_{\omega_1}, \dots , \mathcal{H}_{\omega_n}\}$.
\item If $w_{i_1}, w_{i_2}$ are distinct elements of $\Omega$, then $\{\mathcal{H}_{\omega_{i_1}},\mathcal{H}_{\omega_{i_2}}\}$ is a 
$2$-decomposition of ${\cal U}_{2,\Omega}$.
In particular, $\mathfrak{D} ( {2,\Omega} )=2$.
\end{enumerate}
\end{theorem}

\begin{proof} 
Let $w\in \Omega$.
Since $\Omega\setminus\{\omega\}$ has even size,
the hypergraph ${\cal U}_{2,\Omega\setminus\{\omega\}}$ is a domination hypergraph, and so there exists
a graph $G_0'$ with vertex set $V(G'_0)=\Omega\setminus\{\omega\}$ such that
${\cal U}_{2,\Omega\setminus\{\omega\}}=\DD(G'_0)$.
Let $G_0$ be the join graph $G_0=K_{ \{\omega \} }\vee G'_0$. Then $G_0$ is a graph with vertex set $\Omega$ and minimal dominating sets
$\DD(G_0)=\DD (K_{ \{\omega\} }\vee G'_0)= \big \{\{\omega\}\big \}\cup \DD (G'_0)
=\big \{\{\omega\}\big \}\cup {\cal U}_{2,\Omega\setminus\{\omega\}}={\cal H}_{\omega}$. 
So ${\cal H}_{\omega}$ is a domination hypergraph.

To conclude the proof of the first statement we must demonstrate that if $G$ is a graph with
$\DD(G)=\big \{ \{\omega\}\big \}\cup {\cal U}_{2,\Omega\setminus\{\omega\}}$, then
$G=K_{ \{\omega\} }\vee G'$ for some graph $G'$ with vertex set $\Omega\setminus \{w\}$ and  minimal 
dominating sets $\DD(G')={\cal U}_{2,\Omega\setminus\{\omega\}}$.
Let $G'=G-\omega$ be the graph obtained by deleting the vertex $\omega$ from $G$.
Since $\{ \omega \} \in \DD (G)$, the vertex $\omega $ is universal in $G$ and so
$G=K_{\{ \omega \}} \vee (G-\omega)$.
Moreover, from Lemma~\ref{lema.operacions} we have
$\DD(G)= \big \{ \{\omega\} \big \}\cup \DD (G')$. Thus we conclude that  $\DD(G')={\cal U}_{2,\Omega\setminus\{\omega\}}$.
This completes the proof of the first statement.

Next we are going to prove the second statement; that is, we must demonstrate that $\mathcal{H}_{\omega_1}, \dots , \mathcal{H}_{\omega_n}$ are the minimal
domination completions of ${\cal U}_{2,\Omega}$.

Let $1\leq i \leq n$. It is clear that ${\cal U}_{2,\Omega}\leqslant \big \{\{\omega_i\}\big \}\cup {\cal U}_{2,\Omega\setminus\{\omega_i\}}$; that is,
${\cal U}_{2,\Omega}\leqslant {\cal H}_{\omega_i}$.
Moreover, from statement (1) the hypergraph
${\cal H}_{\omega_i}$ is a domination hypergraph. So
${\cal H}_{\omega_1}, \dots , {\cal H}_{\omega_n}$ are domination completions of ${\cal U}_{2,\Omega}$;
that is, $\{\mathcal{H}_{\omega_{1}}, \dots , \mathcal{H}_{\omega_{n}}\} \subseteq \Dom (2,\Omega)$.

Now let us prove that  $\DDom (2,\Omega) \subseteq \{\mathcal{H}_{\omega_{1}}, \dots , \mathcal{H}_{\omega_{n}}\}$. In order to do this it is enough
to show that if $\mathcal{H}$ is a domination completion of ${\cal U}_{2,\Omega}$,
then there exists $i_0$ such that $\mathcal{H}_{\omega_{i_0}}\leqslant \mathcal{H}$.
Let $\mathcal{H}$ be  a domination completion of ${\cal U}_{2,\Omega}$.
Recall that ${\cal U}_{2,\Omega}$
is not a domination hypergraph.
So ${\cal U}_{2,\Omega}\lneqq \mathcal{H}$ and hence, since the hypergraph ${\cal U}_{2,\Omega}$ consists of
all subsets $A\subseteq \Omega$
of size $|A|=2$, there exists $\omega_{i_0}\in \Omega$ such that $\{ \omega_{i_0} \}\in \mathcal{H}$.
Therefore we have that ${\cal U}_{2,\Omega}\leqslant \mathcal{H}$ and that $\{ \omega_{i_0} \}\in \mathcal{H}$,
and so $\big \{\{\omega_{i_0}\}\big \}\cup {\cal U}_{2,\Omega\setminus\{\omega_{i_0}\}} \leqslant {\cal H}$,
that is, ${\cal H}_{\omega_{i_0}}\leqslant {\cal H}$.

From the above we have that $\DDom (2,\Omega)=\min \{\mathcal{H}_{\omega_{1}}, \dots , \mathcal{H}_{\omega_{n}}\}$.
Observe that if $i\neq j$ then $\mathcal{H}_{\omega_{i}} \not \leqslant \mathcal{H}_{\omega_{j}}$. So,
$\DDom (2,\Omega)=\{\mathcal{H}_{\omega_{1}}, \dots , \mathcal{H}_{\omega_{n}}\}$.
This completes the proof of the second statement.

To complete the proof of the proposition we must prove
that if $\omega_1\neq \omega_2$, then $\{\mathcal{H}_{\omega_1},\mathcal{H}_{\omega_2}\}$ is a $2$-decomposition of ${\cal U}_{2,\Omega}$;
that is, we must demonstrate that
${\cal U}_{2,\Omega}= \mathcal{H}_{\omega_1}\sqcap \mathcal{H}_{\omega_2}$.
Since $\mathcal{H}_{\omega_{i}}=\big \{\{\omega_{i}\}\big \}\cup {\cal U}_{2,\Omega\setminus\{\omega_{i}\}}$,
the union $A_1\cup A_2$ has at least size two whenever
$A_1\in \mathcal{H}_{\omega_1}$ and $A_2\in \mathcal{H}_{\omega_2}$ if $w_1\not= w_2$.
Moreover, it is clear that every subset $\{ \omega_{k} , \omega_{\ell}\}$ with $\omega_{k} \not= \omega_{\ell}$
can be obtained as $A_1\cup A_2$, for some $A_1\in \mathcal{H}_{\omega_1}$ and $A_2\in \mathcal{H}_{\omega_2}$.
Hence we conclude that $\mathcal{H}_{\omega_1}\sqcap \mathcal{H}_{\omega_2}={\cal U}_{2,\Omega}$.
\end{proof}


\subsection{Minimal domination completions of ${\cal U}_{n-1,\Omega}$}\label{sectionn-1n}


From Proposition~\ref{domuniform.caract} we get that if $\Omega$ has size $n\geq 3$, then
the hypergraph ${\cal U}_{n-1,\Omega}$ is not a domination hypergraph. The goal of this subsection is to provide a complete description
of the set $\DDom (n-1,\Omega)$ of the minimal domination completions of ${\cal U}_{n-1,\Omega}$ (Theorem~\ref{thm.nmenysu}), and to 
display their graph realizations (Proposition~\ref{prop.realitzacions}). In addition, we present an upper bound
for the decomposition parameter 
$\mathfrak{D} (n-1,\Omega)$ (Proposition~\ref{thm.dfrac.nmenysu}). Up to now, the computation of the exact value of this 
parameter remains as
an open problem.

In order to prove our results we will use the following five technical lemmas.  
Three of these lemmas are concerned with graphs that are
disjoint union of stars; whereas the other two lemmas
involve some properties of the partial order $\leqslant$.

A tree $T$ of order $n\ge 2$ is a \emph{star} if it is isomorphic to the complete bipartite graph $K_{1,n-1}$.
Observe that a tree $T$ of order $n\ge 2$ is a star if and only if $T$ has at
most one vertex of degree at least 2, the \emph{center} of the star.
If a star $T$ has no vertices of degree at least 2, then $T$ is isomorphic to $K_2$ and both vertices can be considered as 
the center of the star.
Stars can also be characterized as non-empty connected graphs such that all its edges are incident to a \emph{leaf}, 
that is, a vertex of degree 1.
It is clear that every graph without isolated vertices and such that all its edges have at least
one endpoint of degree $1$ is a disjoint union of stars. The following result is a direct consequence of this fact.

\begin{lemma}\label{lem.entornestrelles} 
Let $G$ be a graph without isolated vertices. Then,
$G$ is a disjoint union of stars if and only if  ${\cal N}[G]=E(G)$.
\end{lemma}

\begin{proof}
Suppose first that $G$ is a disjoint union of stars.
If $x$ is a leaf, then $N[x]=\{ x,y \} \in E(G)$,
whereas if $x$ is a vertex of degree $r\ge 2$, then $N[x]=\{ x,y_1,\dots ,y_r \}$  where $y_1,\dots ,y_r$ are the leaves hanging from $x$.
Therefore, we conclude that ${\cal N}[G]=\{ N[x] : x\textrm{ is a leaf } \}$. So, ${\cal N}[G]=E(G)$.

Now suppose that ${\cal N}[G]=E(G)$.
Then, every edge has an endpoint of degree 1, because it is the neighborhood of some vertex. Therefore, $G$ is a disjoint union of stars.
\end{proof}

\begin{lemma} \label{prop.subgrafstars}
Every graph $G$ without isolated vertices contains a spanning subgraph
that is a disjoint union of stars.
\end{lemma}

\begin{proof}
It is sufficient to prove that the statement holds for connected graphs $G$ of order $n\ge 2$.
We proceed by induction on $n$.
The result is trivial for $n=2$. Now assume that $G$ is a connected graph of order $n\ge 3$.
Consider a spanning tree $T$ of $G$. If $T$ is a star, then the result follows.
So we may assume that $T$ is not a star. In such a case 
$T$ has at least two vertices of degree $\ge 2$. Consider an edge of the path joining these two vertices.
By removing this edge, we obtain two trees $T_1$ and $T_2$ of order at least $2$
and without isolated vertices. By inductive hypothesis, both trees contain a spanning subgraph
that is a disjoint union of stars. To finish observe that the union of those subgraphs
is a spanning subgraph of $G$ that is a disjoint union of stars.
\end{proof}

\begin{lemma}\label{lem.domuniostars}
Let $G$ be the disjoint union of the stars $S_1, \dots, S_r$. 
Then, $G$ has exactly $2^r$ minimal dominating sets. Namely, the minimal dominating sets of $G$ are the sets of vertices of the form 
$$\Big \{c_j : j\in J \, \Big \}\cup \Big (\bigcup_{i\in \{1,\dots ,r\}\setminus J} L_i \, \Big )$$
where $J\subseteq \{ 1, \dots ,r \}$,
and where $c_i$ and $L_i$
are respectively the center and the set of leaves of the star $S_i$ (whenever $S_i$
is isomorphic to $K_2$, choose one of the two vertices as the
center and the other as the leaf).
\end{lemma}

\begin{proof}
It is clear that a star $S$ has exactly two minimal dominating sets. Namely,  if the star $S$ is not isomorphic $K_2$, then the minimal dominating sets of $S$
are the set of leaves and the set containing only
the center; whereas if the star $S$ is isomorphic to $K_2$, then the minimal dominating sets of $S$
are the sets containing exactly one vertex.
Now, the result follows by applying Lemma~\ref{lema.operacions} 
because if $G$ is the disjoint union of the stars $S_1,\dots ,S_r$
then $\DD (G)=\{ D_1 \cup \cdots \cup D_r : D_i\in \DD (S_i)\}$.
\end{proof}

\begin{lemma}\label{lem.spanningmenor}
Let $G'$ be a spanning subgraph of $G$. Then, $\DD (G')\le \DD(G)$.
\end{lemma}

\begin{proof}
From $V(G)=V(G')$ and $E(G')\subseteq E(G)$, we have that
every dominating set of $G'$ is also a dominating set of $G$. In particular, if $D'\in \DD(G')$ then $D'$
contains a minimal dominating set $D$ of $G$. Therefore, the inequality $\DD (G')\le \DD(G)$ holds.
\end{proof}

\begin{lemma}\label{lem.hipentorns}
If ${\cal H}$ and ${\cal H}'$ are hypergraphs such that
${\cal H}\le {\cal H}'$, then ${tr}({\cal H}')\le {tr}({\cal H})$.
\end{lemma}

\begin{proof}
Let $X'\in  {tr}({\cal H}')$. We want to prove that there exists $X\in  {tr}({\cal H})$ such that $X\subseteq X'$.
To do this, it is enough to
demonstrate that $X'\cap A\not= \emptyset$ for every $A\in {\cal H}$.
Let $A\in {\cal H}$. Since ${\cal H}\le {\cal H}'$, there exists $A'\in {\cal H}'$ such that $A'\subseteq A$.
By assumption $X'\in  {tr}({\cal H}')$. So $X'\cap A'\not= \emptyset$
and thus $X'\cap A\not= \emptyset$, as we wanted to prove.
\end{proof}

Now, by using these lemmas, we are going to prove the following theorem 
which provides a complete description of all the minimal domination completions of the uniform hypergraph ${\cal U}_{n-1,\Omega}$.

\begin{theorem}\label{thm.nmenysu}
Let $\Omega$ be a finite set of size $n\geq 3$. Then, the minimal domination completions of ${\cal U}_{n-1,\Omega}$ are the domination hypergraphs
${\cal H}$ of the form ${\cal H}=\DD(G)$ where $G$ is a disjoint union of stars; that is, 
$$\DDom (n-1,\Omega) =\{ \DD(G) : G\textrm{ is a disjoint union of stars with vertex set } \Omega \}.$$
\end{theorem}

\begin{proof} 
Let $\Sigma=\{G \, : \, G\textrm{ is a disjoint union of stars with vertex set } \Omega \}$.
First we will prove that ${\cal U}_{n-1,\Omega}\le \DD (G)$ for all $G\in \Sigma$; that is, we must demonstrate that if
$A\in {\cal U}_{n-1,\Omega}$ then there exists $D\in \DD(G)$ such that $D\subseteq A$.
So,  let $A\in {\cal U}_{n-1,\Omega}$. Then $A=\Omega\setminus \{ \omega_{0} \}$
for some $\omega_{0} \in \Omega$.
By Lemma~\ref{lema.unio}, there exists a minimal dominating set $D_{0}$ 
of $G$ not containing $\omega_{0}$. Thus, $D_{0}\subseteq \Omega\setminus \{ \omega_{0} \}$. So we can set $D=D_0$.

Now, we will prove that if ${\cal H}$ is a domination completion of ${\cal U}_{n-1,\Omega}$, then 
there exists $G\in \Sigma$ such that $\DD(G)\leqslant {\cal H}$.
So, let ${\cal H}$ be a domination completion of ${\cal U}_{n-1,\Omega}$. Then ${\cal U}_{n-1,\Omega}\leqslant {\cal H}$
and there is a graph $G_{\cal H}$ with vertex set $\Omega$
such that ${\cal H}= \DD (G_{\cal H})$.
Notice that if ${\cal U}_{n-1,\Omega}\le \DD (G_{\cal H})$, then  $G_{\cal H}$ has no isolated vertices, (because otherwise the
isolated vertex $\omega_{0}$ should be at every minimal dominating set of $G_{\cal H}$ implying
that $\Omega \setminus \{ \omega_{0} \} \in {\cal U}_{n-1,\Omega}$ does not contain any minimal dominating set of $G_{\cal H}$, which is a contradiction).
Thus, by Lemma~\ref{prop.subgrafstars}, there exists a spanning subgraph $G$ of
$G_{\cal H}$ that is a disjoint union of  stars. Since $G$ is a spanning subgraph of 
$G_{\cal H}$, by Lemma~\ref{lem.spanningmenor} it follows that 
$\DD (G)\le \DD (G_{\cal H})$. Therefore we conclude that $G\in \Sigma$ and $\DD(G)\leqslant {\cal H}$.

Finally, it remains to prove that the dominating hypergraphs of distinct disjoint union of stars with vertex set 
$\Omega$ are either equal or non-comparable. In other words, we must demonstrate that if $\DD (G)\le \DD (G')$ with $G,G'\in \Sigma$,
then $G=G'$. So, let $G,G'\in \Sigma$ with $\DD (G)\le \DD (G')$. Then,
from Lemma~\ref{entorn.prop} and Lemma~\ref{lem.hipentorns} it follows that ${\cal N}[G']={tr}(\DD (G'))\le {tr}(\DD (G))={\cal N}[G]$.
By applying Lemma~\ref{lem.entornestrelles} we get that ${\cal N}[G']=E(G')$ and ${\cal N}[G]=E(G)$. Therefore
$E(G') \le E(G)$. Hence $E(G') \subseteq E(G)$ because $E(G')$ and $E(G)$ are 2-uniform hypergraphs.
At this point observe that the addition of an edge to a graph that is a disjoint union of
stars gives rise to a graph not satisfying this property. Therefore we conclude that  $E(G)=E(G')$ and, consequently, $G=G'$.
\end{proof}

The following proposition characterizes all graphs that realize a  minimal domination completion of ${\cal U}_{n-1,\Omega}$.
After its proof we present an example of a minimal domination completion ${\cal H}_0$  of the uniform hypergraph ${\cal U}_{n-1,\Omega}$
whenever $n=8$, as well as the description of all the graph realizations of ${\cal H}_0$ (the example is illustrated
in Figure~\ref{fig.mateixos}).

\begin{figure}
\begin{center}
\includegraphics [width=0.8\textwidth]{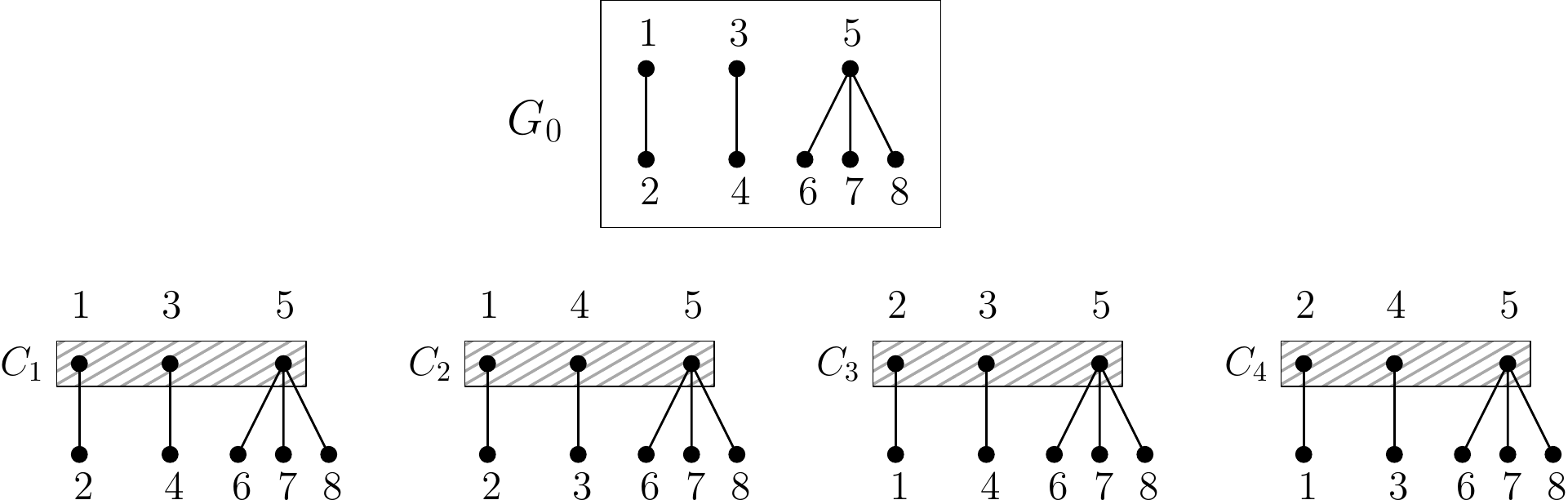}
\vspace{7mm}

\includegraphics [width=\textwidth] {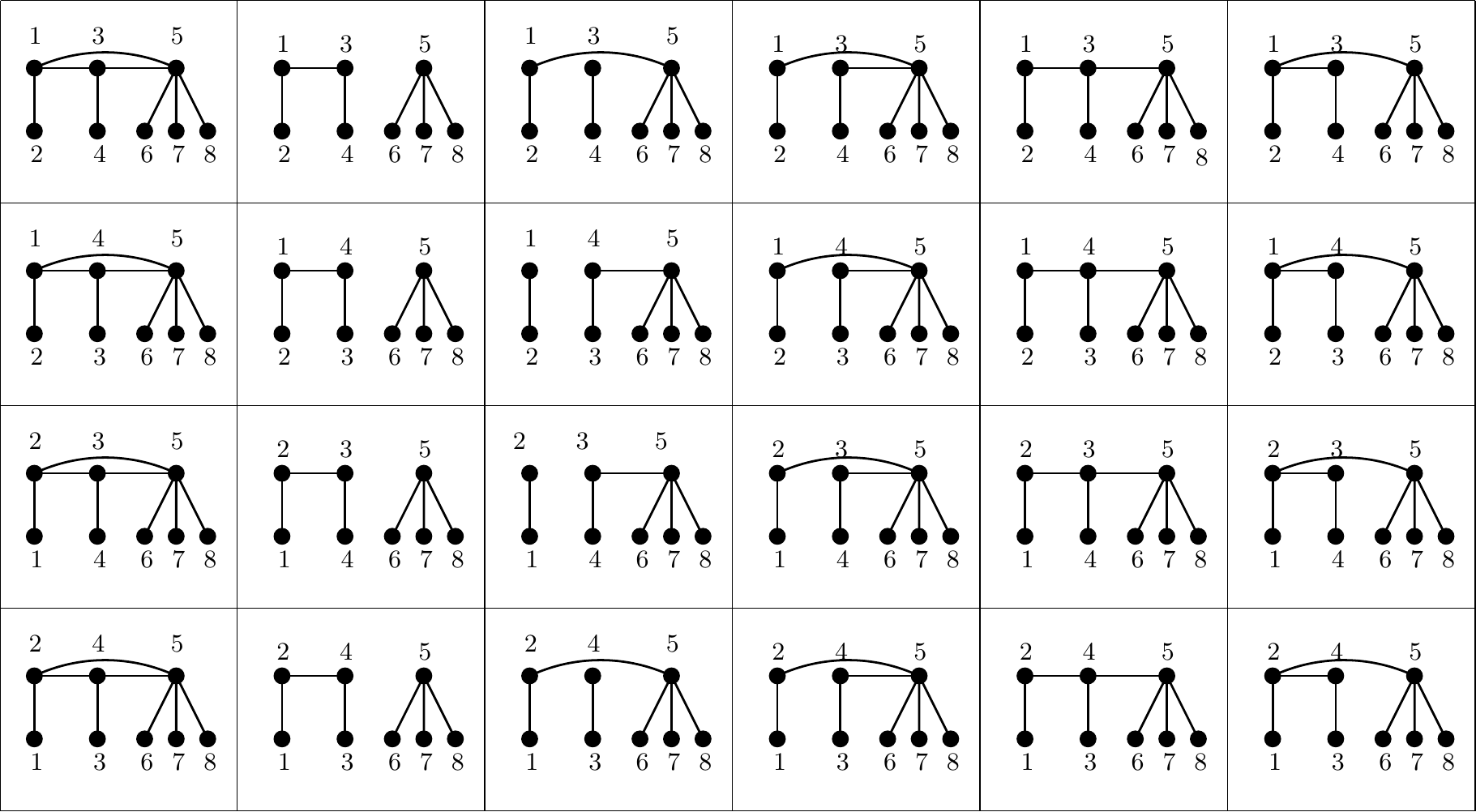}

\caption{The graph $G_0$ together with the 24 graphs obtained by adding edges joining vertices of  $C_i$ (where $i=1,2,3,4$)
are all the 25 graph realizations of the domination hypergraph ${\cal H}_0$ of Example~\ref{exemple78}
(the hypergraph  ${\cal H}_0=\DD (G_0)$ is a minimal domination completion of the uniform hypergraph
${\cal U}_{7,\Omega}$ where $\Omega =\{1,2,3,4,5,6,7,8\}$).}\label{fig.mateixos}
\end{center}
\end{figure}

\begin{proposition}\label{prop.realitzacions}
Let $G$ be a graph with vertex set $\Omega$ that is a disjoint union of stars, and let $G'$ be a graph with vertex set $\Omega$. 
Then, $\DD (G)=\DD (G')$ if and only if $G'$ is any graph that can be obtained from $G$ in the following way:
choosing a set  $C$ formed by exactly one center of every connected component of $G$ and
adding to $G$ any set of edges joining vertices of $C$.
\end{proposition}

\begin{proof}
Let $G'$ be a graph with vertex set $\Omega$.
By applying Lemma~\ref{entorn.prop} and Lemma~\ref{lem.entornestrelles} we get that $\DD (G')=\DD (G)$ if and only if ${\cal N}[G']={\cal N}[G]$
if and only if ${\cal N}[G']=E(G)$.
It is not hard to prove that  ${\cal N}[G']=
E(G)$ if and only if the following two conditions are satisfied: $E(G)\subseteq E(G')$, and for every edge
$\{ x,y \}\in E(G)$  either $x$ or $y$ has degree 1 in $G'$.
Therefore we conclude that  $\DD (G')=\DD (G)$ if and only if $G'$ is obtained from $G$
by adding edges joining vertices of a set containing exactly one center of each star of $G$.
\end{proof}

\begin{example}\label{exemple78}
Let $\Omega =\{1,2,3,4,5,6,7,8\}$.
By Theorem~\ref{thm.nmenysu}, the minimal domination completions of ${\cal U}_{7,\Omega}$ are the hypergraphs of the form
$\DD(G)$ where  $G$ is a disjoint union of stars with vertex set $\Omega$.
It is not difficult to check that there are 5041 such graphs $G$, all of them providing different domination hypergraphs.
Therefore, $|\DDom (7,\Omega) |=5041$.
One of these graphs $G$ is the graph $G_0$ obtained as the disjoint union of 3 stars, 
two of them isomorphic to $K_2$ and the other one, isomorphic to $K_{1,3}$; namely, the graph $G_0$ with edge set
$E(G_0)=\{ \{ 1,2\}, \{3,4\}, \{5, 6 \}, \{5, 7 \}, \{5, 8 \}  \}$.
From Lemma~\ref{lem.domuniostars}, this graph $G_0$ has the following $2^3=8$ minimal dominating sets:
\begin{align*}
\DD (G_0)=\{ &\{  1,3,5 \}, \{ 1,3,6,7,8  \}, \{ 1,4,5  \}, \{  1,4,6,7,8 \}, \\
&\{  2,3,5 \}, \{ 2,3,6,7,8  \}, \{ 2,4,5  \}, \{  2,4,6,7,8 \} \}.
\end{align*}
So, the hypergraph ${\cal H}_0=\DD (G_0)$ is a minimal domination completion of ${\cal U}_{7,\Omega}$.
In order to obtain all the graph realizations of ${\cal H}_0$, we apply Proposition~\ref{prop.realitzacions}.
In this case we have four possibilities for the set $C$ containing exactly one center of each star.
Concretely $C$ is either $C_1=\{ 1,3,5 \}$, or $C_2=\{ 1,4,5 \}$, or $C_3=\{ 2,3,5 \}$, or $C_4=\{ 2,4,5 \}$.
The graphs $G'$ such that its collection of minimal dominating sets is 
$\DD(G')=\DD(G_0)$ are obtained by fixing one of the sets $C_i$ and adding edges joining vertices of  $C_i$.
It is easy to check that there are exactly 24 different graphs $G'\not= G_0$ obtained in this way (see Figure~\ref{fig.mateixos}).
\end{example}

To conclude this subsection we present an upper bound on the decomposition parameter $\mathfrak{D} (n-1,\Omega)$ of the uniform hypergraph 
${\cal U}_{n-1,\Omega}$ where $|\Omega|=n$ (Proposition~\ref{thm.dfrac.nmenysu}).
It is worth noting that an exhaustive analysis of all possible cases shows that the equality holds whenever $2\le n\le 5$. 
However, it remains an open problem to determine if the equality holds for $n\ge 6$.

\begin{proposition}~\label{thm.dfrac.nmenysu}
Let $\Omega$ be a finite set of size $n\geq 3$. Then, $\mathfrak{D} (n-1,\Omega)\le n-1$; that is,
there are $n-1$ minimal domination completions 
${\cal H}_1, \dots, {\cal H}_{n-1}$ of ${\cal U}_{n-1,\Omega}$ such that
${\cal U}_{n-1,\Omega}={\cal H}_1 \sqcap \dots \sqcap {\cal H}_{n-1}$. 
\end{proposition}

\begin{proof}
Let $\Omega =\{ \omega_1,\dots ,\omega_n\}$. For $1\le i \le n-1$, let ${\cal H}_i$ be the domination hypergraph ${\cal H}_i=\DD(S_i)$ where $S_i$ is 
the star with center $\omega_i$ and isomorphic to $K_{1,n-1}$. By Theorem~\ref{thm.nmenysu}, the hypergraphs
${\cal H}_1, \dots, {\cal H}_{n-1}$ are minimal domination completions 
of ${\cal U}_{n-1,\Omega}$. Let us show that ${\cal U}_{n-1,\Omega}={\cal H}_1 \sqcap \dots \sqcap {\cal H}_{n-1}$. 
It is clear that ${\cal H}_i=\{D_{i,1},D_{i,2}\}$ where $D_{i,1}=\{ \omega_i \}$ and
$D_{i,2}= \Omega \setminus \{  \omega_i \}$. On the one hand, the elements of ${\cal H}_1 \sqcap \dots \sqcap {\cal H}_{n-1}$ have size at least $n-1$
and hence the inequality ${\cal H}_1 \sqcap \dots \sqcap {\cal H}_{n-1} \leqslant {\cal U}_{n-1,\Omega}$ holds.
On the other hand,  if $A\in {\cal U}_{n-1,\Omega}$, then  $A=\Omega \setminus \{ \omega \}$ for some $w\in \Omega$,
and thus we get that: $A=D_{1,1} \cup \cdots \cup D_{n-1,1}$ if $w=w_n$;
whereas $A=D_{i_0,2} \cup \big( \cup_{i\not= i_0} D_{i,1} \big)$ if $w=w_{i_0}\neq w_n$.
So, the inequality  ${\cal U}_{n-1,\Omega} \leqslant {\cal H}_1 \sqcap \dots \sqcap {\cal H}_{n-1}$ also holds.
Therefore, since $\leqslant$ is a partial order, we conclude that ${\cal U}_{n-1,\Omega}={\cal H}_1 \sqcap \dots \sqcap {\cal H}_{n-1}$.
\end{proof}


\subsection{Minimal domination completions of ${\cal U}_{r,\Omega}$ whith $|\Omega|\leq 5$}\label{section135}


The aim of this subsection is to determine the set 
$\DDom (r,\Omega)$ of the minimal domination completions 
of the uniform hypergraph ${\cal U}_{r,\Omega}$ where $\Omega$ is a finite set of size $|\Omega|=n \nobreak \leq \nobreak 5$ and $1\leq r \leq n$.
From Proposition~\ref{domuniform.caract} and Theorem~\ref{thcompletacio} we get that $\DDom (r,\Omega)=\{{\cal U}_{r,\Omega}\}$ if and only if 
$(r,n)\neq(2,3),(2,5),(3,4),(3,5),(4,5)$. In addition, the results of the the preceding subsections provide a complete description of the set $\DDom (r,\Omega)$
whenever $(r,n)=(2,3),(2,5),(3,4),(4,5)$. Therefore, it only remains to determine the set of minimal domination completions $\DDom (r,\Omega)$
whenever $(r,n)=(3,5)$. 

This subsection deals with this issue. Namely, the goal of this subsection is to prove that,
for $r=3$ and $n=5$ the uniform hypergraph ${\cal U}_{r,\Omega}$ has $22$ minimal domination completions: 12 of the form $\DD(G)$ with $G$ isomorphic
to a cycle $C_5$, and 10 of the form $\DD(G)$ with $G$ isomorphic
to the complete bipartite graph $K_{2,3}$. This result is stated in Theorem~\ref{thm.ComplMiniU35}. 
This and all the other results about the minimal domination completions of the uniform hypergraphs  ${\cal U}_{r,\Omega}$ where  $1\leq r \leq |\Omega |\le 5$,
are summarized in Figure~\ref{fig.taulaU35}
(in each case, the graphs $G$ in the figure
provide the realization of all the minimal domination completions ${\cal H}$ of ${\cal U}_{r,\Omega}$; namely,
${\cal H}\in \DDom(r,\Omega)$ if and only if ${\cal H}= \DD(G')$ for some graph  $G'$ isomorphic to a graph $G$  in the figure).
\begin{figure}
\begin{center}
\includegraphics [width=\textwidth]{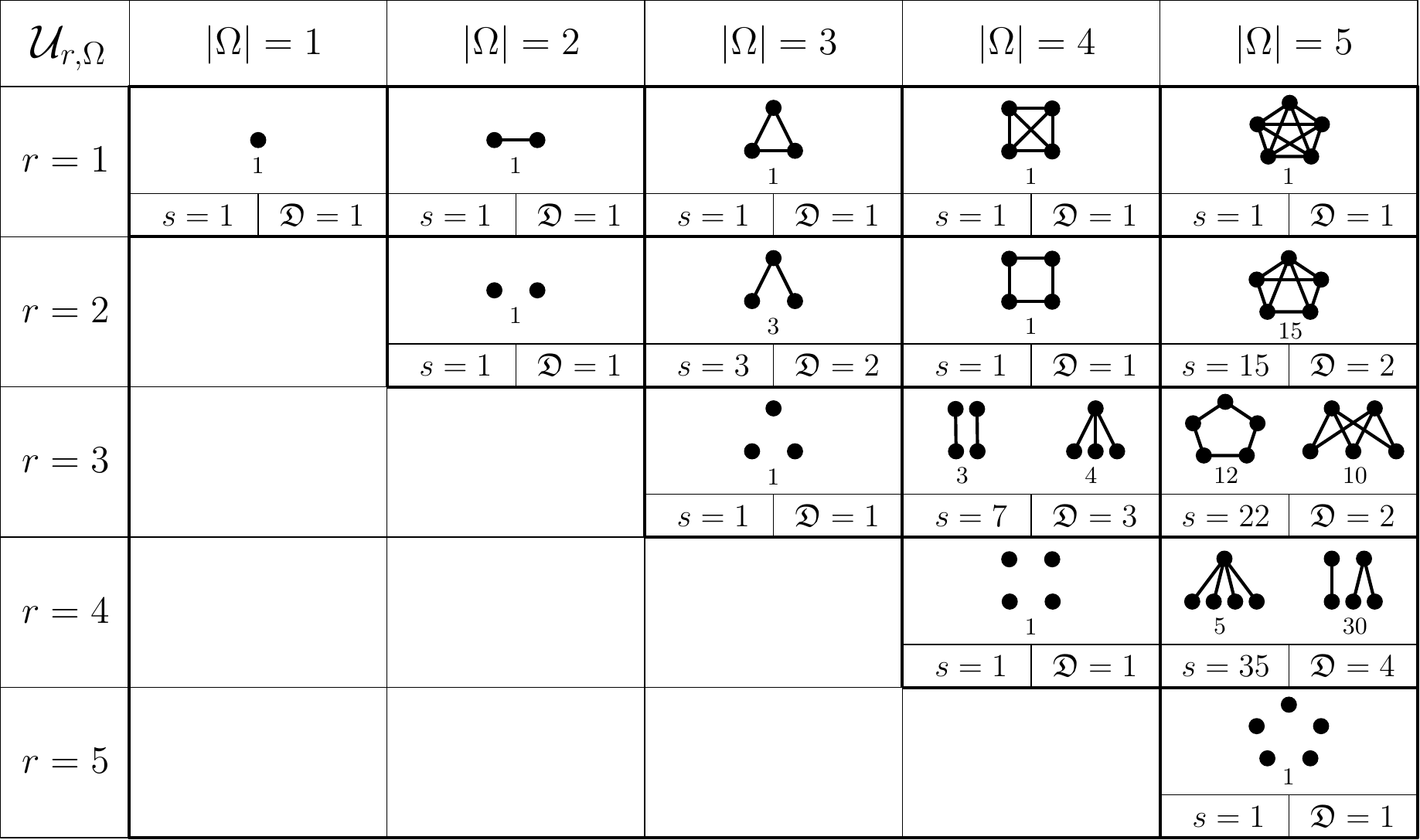}
\caption{For $1\leq n=|\Omega|\leq 5$, the set of the minimal domination completions of 
${\cal U}_{r,\Omega}$ is $\DDom(r,\Omega)=\{ \DD(G') \, : \, G'$ is a graph isomorphic to a graph $G$ in the figure$\}$. The
number below each graph $G$ denotes the number of different hypergraphs ${\cal H}\in Dom(r,\Omega)$ with ${\cal H}=\DD (G')$ for some graph with $G'$ isomorphic to $G$.
In addition, in each case,  the total number $s=|\DDom(r,\Omega)|$ of minimal domination completions of  ${\cal U}_{r,\Omega}$,
and the decomposition parameter $\mathfrak{D}=\mathfrak{D} (r,\Omega)$ are given.}
\label{fig.taulaU35}
\end{center}
\end{figure}

\begin{theorem}\label{thm.ComplMiniU35}
Let $\Omega$ be a finite set of size $|\Omega|=5$. Let ${\cal C}_5$ and  ${\cal K}_{2,3}$ be the families of 
graphs with vertex set $\Omega$, where the graphs of
${\cal C}_5$ are exactly those isomorphic to the
cycle $C_5$, whereas the graphs of ${\cal K}_{2,3}$ are all those isomorphic to the complete bipartite graph $K_{2,3}$.
The following statements hold:
\begin{enumerate}
\item The minimal domination completions of ${\cal U}_{3,\Omega}$ are the domination hypergraphs ${\cal H}$ of the form
${\cal H}=\DD (G)$ where the graph $G$ is isomorphic to either a cycle $C_5$ or to a complete bipartite graph $K_{2,3}$; that is, 
$\DDom (3,\Omega) =\{ \DD(G) : G\in {\cal C}_5 \cup {\cal K}_{2,3}\}.$
\item The uniform hypergraph ${\cal U}_{3,\Omega}$ has $22$ minimal domination completions; that is, \allowbreak $|\DDom (3,\Omega)|=22$.
\newpage
\item The uniform hypergraph ${\cal U}_{3,\Omega}$ has decomposition parameter $\mathfrak{D} (3,\Omega)=2$; that is,
there exist minimal domination completions 
${\cal H},  {\cal H}'$ of ${\cal U}_{3,\Omega}$ such that
${\cal U}_{3,\Omega}={\cal H}\sqcap {\cal H}'$. 
\end{enumerate}
\end{theorem}

The rest of this subsection is devoted to prove this theorem. 
From now on we set $\Omega =\{ 1,2,3,4,5 \}$.

First observe that if $G\in {\cal C}_5$, then $\DD(G)$ contains the five pairs of non-adjacent vertices; 
while  if $G\in {\cal K}_{2,3}$, then $\DD (G)$ contains both stable sets and the 6 pairs of adjacent vertices (see Figure~\ref{fig.C5K23}).
Using these facts it is easy to check that if $G,G' \in {\cal C}_5 \cup {\cal K}_{2,3}$, then $\DD(G)=\DD(G')$ if and only if $G=G'$.
Therefore, $|\{ \DD(G) : G\in {\cal C}_5 \cup {\cal K}_{2,3}\}|=\allowbreak
|{\cal C}_5|+|{\cal K}_{2,3}|=12+10=22$. Thus, the statement (2) of the theorem follows from the first one.
\begin{figure}[h]
\begin{center}
\includegraphics [width=\textwidth]{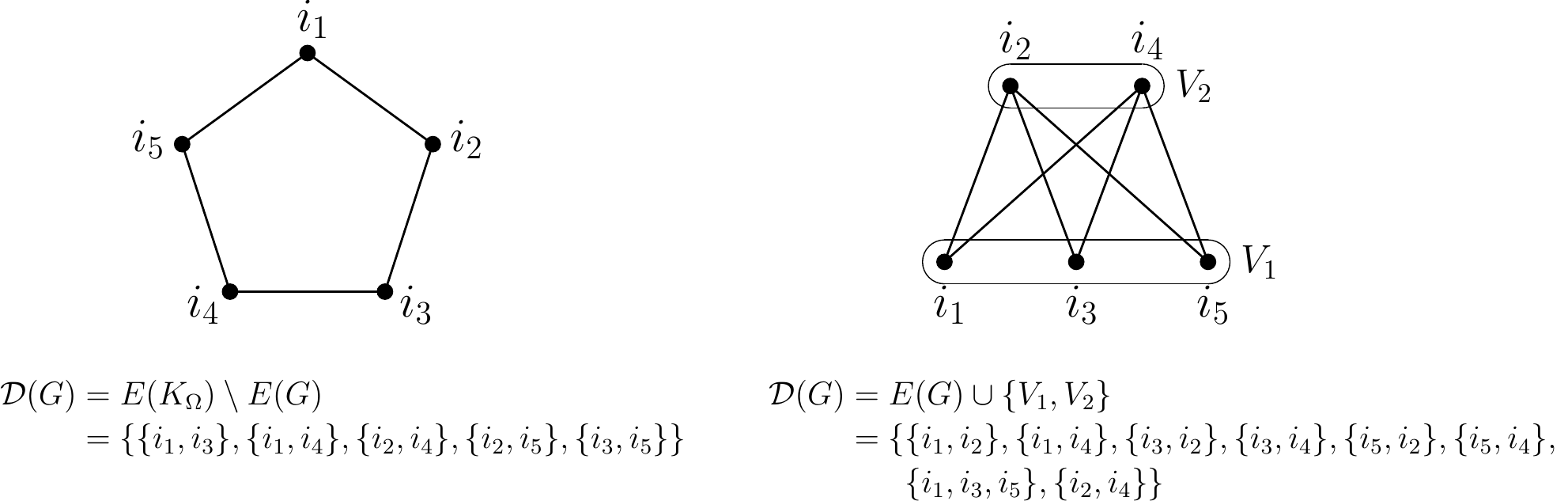}
\caption{Minimal dominating sets of a graph $G$ with vertex set $V(G)=\{ i_1,i_2,i_3,i_4,i_5\}$,
isomorphic to $C_5$ (left) and isomorphic to $K_{2,3}$ (right).}\label{fig.C5K23}
\end{center}
\end{figure}

Now let us demonstrate the third statement of the theorem. The statement (1) will be proved after doing this.

Recall that ${\cal U}_{3,\Omega}$ is not a domination hypergraph (Proposition~\ref{domuniform.caract}). So,
$\mathfrak{D} (3,\Omega)\geq 2$. The inequality $\mathfrak{D} (3,\Omega)\leq 2$ follows from Proposition~\ref{prop.dfraku35}. This proposition
shows all the ways to obtain ${\cal U}_{3,\Omega}$ as $\DD (G_1)\sqcap \DD (G_2)$, when $G_1,G_2\in{\cal C}_5 \cup {\cal K}_{2,3}$.

\begin{proposition}\label{prop.dfraku35}
Let $G_1,G_2\in {\cal C}_5 \cup {\cal K}_{2,3}$. Then,
${\cal U}_{3,\Omega}=\DD (G_1) \sqcap \DD (G_2)$ if and only if $G_1,G_2\in {\cal C}_5$  and $E(G_1)\cup E(G_2)=E (K_{\Omega})$.
\end{proposition}

\begin{proof}
First consider the case $G_1,G_2\in {\cal C}_5$  with $E(G_1)\cup E(G_2)=E (K_{\Omega})$.
In such a case 
$E(G_1)\cap E(G_2)=\emptyset$, and so we get that 
$\DD (G_1)=E(G_2)$ and $\DD (G_2)=E(G_1)$.
Thus, every set $A_1 \cup A_2$, with  $A_1\in \DD (G_1)$ and $A_2 \in \DD (G_2)$, has size 3 or 4.
It is straightforward to check that every element of ${\cal U}_{3,\Omega}$ can be obtained as $A_1 \cup A_2$ where $A_1\in \DD (G_1)$ and $A_2 \in \DD (G_2)$.
Therefore, ${\cal U}_{3,\Omega}=\DD (G_1) \sqcap \DD (G_2)$.

The proof of the proposition will be completed by showing that, in any other case, there exists $A_0\in \DD (G_1)\cap \DD (G_2)$ with $|A_0|=2$.
Indeed, if there exists $A_0 \in \DD (G_1)\cap \DD (G_2)$, then $A_0\in \DD (G_1) \sqcap \DD (G_2)$, and so 
$\DD (G_1) \sqcap \DD (G_2)$ has at least an element of size  $|A_0|$. Hence $\DD (G_1) \sqcap \DD (G_2)\not= {\cal U}_{3,\Omega}$ 
because $|A_0|=2$.

Therefore, we must demonstrate that there exists $A_0 \in \DD (G_1)\cap \DD (G_2)$ with $|A_0|=2$. We distinguish three cases:
whenever $G_1,G_2\in {\cal C}_5$  and $E(G_1)\cup E(G_2)\not= E (K_{\Omega})$;
whenever $G_1,G_2\in {\cal K}_{2,3}$; and whenever $G_1\in {\cal C}_{5}$ and $G_2\in {\cal K}_{2,3}$.
If $G_1,G_2\in {\cal C}_5$  and $E(G_1)\cup E(G_2)\not= E (K_{\Omega})$, then there exists 
$\{x,y \}\in  E (K_{\Omega})\setminus (E(G_1)\cup E(G_2))$. In this case $\DD(G_i)=E (K_{\Omega})\setminus E(G_i)$. So, we can set 
$A_0=\{x,y\}\in \DD (G_1)\cap \DD (G_2)$.
Now let us assume that $G_1,G_2\in {\cal K}_{2,3}$. Then $|E(G_1)|=|E(G_2)|=6$. So $E(G_1)\cap E (G_2)\not= \emptyset$ and thus there exists 
$\{x,y \}\in  E(G_1)\cap E (G_2)$. Since $E(G_i)\subseteq \DD(G_i)$, in this case the subset $A_0=\{x,y\}$ satisfies the required conditions.
Finally, suppose that $G_1\in {\cal C}_5$ and $G_2\in {\cal K}_{2,3}$. Then $|E(K_{\Omega})\setminus E(G_1)|$=5 and $|E(G_2)|=6$. 
So there exists $\{ x,y \}\in (E(K_{\Omega})\setminus E(G_1))\cap E (G_2)\subseteq \DD (G_1) \cap \DD (G_2)$, and thus the proof is completed by setting
$A_0=\{x,y\}$.
\end{proof}

At this point, the proof of Theorem~\ref{thm.ComplMiniU35} will be completed by proving the first statement. This statement follows as a consequence of
Propositions~\ref{prop.c5k23domcompl},~\ref{prop.c5k23nocomp} and~\ref{prop.c5k23minimals}. 
The first two propositions show that the hypergraphs $\DD (G)$, where $G\in {\cal C}_5 \cup {\cal K}_{2,3}$, are domination completions of 
${\cal U}_{3,\Omega}$, and that any pair of different such hypergraphs are non-comparable;
whereas the last proposition states that 
the minimal domination completions of  ${\cal U}_{3,\Omega}$ are hypergraphs of the form $\DD (G)$, where $G\in  {\cal C}_5 \cup {\cal K}_{2,3}$.
The proof of Proposition~\ref{prop.c5k23minimals} is involved and requires three 
technical lemmas concerning the size of the elements of the minimal domination completions of ${\cal U}_{3,\Omega}$ and their transversal
(Lemmas~\ref{lem.sizetransversal},~\ref{lem.size3} and~\ref{lem.size2}).

\begin{proposition}\label{prop.c5k23domcompl}
If $G\in {\cal C}_5 \cup {\cal K}_{2,3}$, then ${\cal U}_{3,\Omega}\le \DD (G)$.
\end{proposition}

\begin{proof}
From Lemma~\ref{lem.ordre}, we must demonstrate that if $A$ is a subset of $\Omega$ of size 3, then there exists $D\in \DD (G)$ such that $D\subseteq A$.
This is clear if $G\in  {\cal C}_5$, because in such a case
every set of three vertices of $G$  contains a pair of two non-adjacent vertices,
that are a minimal dominating set of $G$.
Now let assume that $G\in {\cal K}_{2,3}$. In this case the result follows by taking into account that the stable set of size 3
is a minimal dominating set of $G$, and that every other set of three vertices contains two adjacent vertices.
So any subset of size three contains a minimal dominating set of $G$.
\end{proof}

\begin{proposition}\label{prop.c5k23nocomp}
Let $G_1,G_2\in {\cal C}_5 \cup {\cal K}_{2,3}$. If $\DD (G_1)\leq \DD (G_2)$, then $\DD (G_1) = \DD (G_2)$.
\end{proposition}

\begin{proof}
First, suppose that $G_1,G_2\in {\cal C}_5$. Then the hypergraphs $\DD (G_1)$ and $\DD (G_2)$ contain both 
exactly five elements of size 2. In such a case  it is clear that if $\DD (G_1)\leq \DD (G_2)$, then  $\DD (G_1)=\DD (G_2)$.

Now assume that $G_1,G_2\in {\cal K}_{2,3}$. Then the hypergraphs $\DD (G_1)$ and $\DD (G_2)$ contain both exactly one element of size 3 and 7 elements of size 2. 
Therefore, from $\DD (G_1)\leq \DD (G_2)$ we deduce that the 7 elements of size 2 must be the same. In addition,
since there is no inclusion relation between the elements of a hypergraph, the element of size 3 must be also the same. Hence we conclude that
$\DD (G_1)=\DD (G_2)$.

The proof will be completed by showing that the inequality $\DD (G_1)\leq \DD (G_2)$ is not possible neither in the case $G_1\in  {\cal K}_{2,3}$ and $G_2\in {\cal C}_5$,
nor in the case $G_1\in {\cal C}_5$ and  $G_2\in {\cal K}_{2,3}$.
If $G_1\in  {\cal K}_{2,3}$ and $G_2\in {\cal C}_5$, then $\DD(G_1)$ contains 6 elements of size 2 while $\DD (G_2)$ contains only 5 elements of size 2.
Thus, in such a case, the inequality $\DD (G_1)\leq \DD (G_2)$ is not possible.
Finally, assume $G_1\in {\cal C}_5$ and  $G_2\in {\cal K}_{2,3}$. 
If  $\DD (G_1)\leq \DD (G_2)$,  then the 5 elements of size 2 of $\DD (G_1)$ must be in $\DD (G_2)$; that is, 
$\DD (G_1)\subseteq \DD (G_2)$. But the five elements of $\DD (G_1)$ correspond to 5 
pairs of non-adjacent vertices of a cycle of order 5, that induce also a cycle of order 5. Nevertheless, there is not possible to induce a cycle 
of order 5 with 5 elements of size 2 of $\DD (G_2)$. Therefore, $\DD (G_1)\leq \DD (G_2)$ is not possible in that case. This completes the
proof of the proposition.
\end{proof}

\begin{lemma}\label{lem.sizetransversal}
Let ${\cal H}$ be a hypergraph. If ${\cal U}_{3,\Omega}\leqslant {\cal H}$, 
then $|X|\ge 3$ for every $X\in {tr} ({\cal H})$.
\end{lemma}

\begin{proof}
On the contrary, assume that there exists  $X\in {tr} ({\cal H})$ such that $|X|\le 2$.
In such a case, consider a subset $A\subseteq \Omega$ of size $|A|=3$ satisfying $A\cap X=\emptyset$.
Since $|A|=3$, hence $A\in {\cal U}_{3,\Omega}$. Therefore there exists $B\in {\cal H}$ contained in $A$ because ${\cal U}_{3,\Omega}\leqslant {\cal H}$.
Hence $B\cap X\subseteq A\cap X$, and so $ B\cap X=\emptyset$. This leads us to a contradiction because $B\in {\cal H}$ and $X\in {tr} ({\cal H})$.
\end{proof}

\begin{lemma}\label{lem.size3}
Let ${\cal H}$ be  a minimal domination completion of  ${\cal U}_{3,\Omega}$.
If there exists $A\in {\cal H}$ such that $|A|=3$,
then  ${\cal H}= \DD (G)$ for some  graph $G\in {\cal K}_{2,3}$.
\end{lemma}

\begin{proof}
It is enough to prove that $\DD (G) \le {\cal H}$ for some $G\in {\cal K}_{2,3}$,
because by Proposition~\ref{prop.c5k23domcompl}, ${\cal U}_{3,\Omega}\le \DD (G)$,
and so, the minimality of  ${\cal H}$ implies that $\DD (G) = {\cal H}$.

Without loss of generality we may assume that $A=\{ 1,2,3 \}\in{\mathcal H}$.
In such a case, $\{ 1,4,5 \}$, $\{ 2,4,5 \}$ and $\{ 3,4,5 \}$ are in $tr({\cal H})$,
because all these subsets have non-empty intersection with the elements of ${\cal H}$, 
and there are no elements of cardinality less or equal than 2 in $tr ({\cal H})$
(Lemma~\ref{lem.sizetransversal}).
Since ${\cal H}$ is a domination hypergraph, there exists a graph $G_0$ such that
${\cal H}=\DD (G_0)$, and so
$tr({\cal H})={\cal N}[G_0]$ (Lemma~\ref{entorn.prop}).
Therefore, $\{ 1,4,5 \}$, $\{ 2,4,5 \}$ and $\{ 3,4,5 \}$ are the closed neighborhoods for some $x,y,z\in \Omega$; that is,
$N_{G_0}[x]=\{ 1,4,5 \}$, $N_{G_0}[y]=\{ 2,4,5 \}$ and $N_{G_0}[z]=\{ 3,4,5 \}$.
Observe that at least one of the elements $x,y,z\in \Omega$ is different from $4,5$. So, 
without loss of generality we may assume that $x\neq 4,5$ and so $x=1$.
Thus $\{ 1,4,5 \}=N_{G_0}[1]$. Hence $\{1,4\},\{1,5\}\in E({G_0})$, and consequently, 
$N_{G_0}[4]\neq \{ 2,4,5 \},\{ 3,4,5 \}$ and $N_{G_0}[5]\neq \{ 2,4,5 \}, \{ 3,4,5 \}$.
So we conclude that $N_{G_0}[1]=\{ 1,4,5 \}$,  that $N_{G_0}[2]=\{ 2,4,5 \}$, and  that $N_{G_0}[3]=\{ 3,4,5 \}$. Hence it follows that 
$F=\{ \{ 1,4 \}, \{ 1,5 \}, \{ 2,4 \},\{ 2,5 \}, \{ 3,4 \} , \{ 3,5 \} \} \subseteq E({G_0}) $.
At this point let us consider the subgraph $G$ induced by the edges of $F$. 
Observe that $G$ is isomorphic to $K_{2,3}$ with stable sets $\{1,2,3\}$ and $\{ 4,5 \}$. So,
$G\in {\cal K}_{2,3}$. Moreover, the graph $G$ is a 
spanning subgraph of $G_0$ and hence, from Lemma~\ref{lem.spanningmenor} it follows that $\DD (G) \le \DD (G_0)={\cal H}$.
This completes the proof of the lemma.
\end{proof}

\begin{lemma}\label{lem.size2}
Let ${\cal H}$ be  a minimal domination completion of  ${\cal U}_{3,\Omega}$.
If $|A|=2$ for all $A\in {\cal H}$,
then  ${\cal H}= \DD (G)$ for some  graph $G\in {\cal C}_{5}$.
\end{lemma}

\begin{proof}
Let $G_0$ be a graph with ${\cal H}=\DD(G_0)$.
Reasoning as in the proof of the previous lemma,
here it is enough to prove that $\DD (G) \le {\DD (G_0)}$ for some $G\in {\cal C}_{5}$.

To prove this inequality we will use the following four facts.
 
First, notice that $G_0$ has no vertex of degree $4$, because otherwise there would be an element in $\DD(G_0)$ of size 1.

Secondly, we claim that if $\{ a,b \}\notin {\cal H}$, then  $\{ a,b \}\in E({G}_0)$. 
Let us prove our claim. Suppose to the contrary that $a$ and $b$ are non-adjacent in $G_0$. 
In such a case, both vertices $a$ and $b$ belong to a minimal dominating set 
$D$ of $\DD (G_0)$ (for instance, we can consider a maximal independent set $D$ containing 
$a$ and $b$). But $\DD (G_0)={\cal H}$ and, by assumption,  
all the elements of ${\cal H}$ have size 2. Therefore we conclude 
that $\{ a,b \}=D\in \DD (G_0)={\cal H}$, which is a contradiction. This completes the proof of our claim.

Next,  observe that ${\cal N} [G_0]=tr (\DD (G_0))=tr ({\cal H})$ (Lemma~\ref{entorn.prop}).
So, by applying Lemma~\ref{lem.sizetransversal} it follows that all the elements of ${\cal N} [G_0]$ have at least 3 elements.

Finally, let us show that in fact ${\cal N} [G_0]$  has at least one element $X$ of size 3.
Suppose on the contrary that it is not true.
If  $ {\cal N} [G_0]=\{ \Omega \}$, then $G_0$ has at least one vertex of degree $4$, which is not possible.
So, without loss of generality we may assume that $\Omega\notin {\cal N} [G_0]$ and that $\{ 1,2,3,4 \}\in {\cal N} [G_0]$.
In such a case,  all subsets of cardinality 4 must be in ${\cal N} [G_0]$,
because otherwise there exists  $ j \in \bigcap_{N\in {\cal N}[G_0]} N $,
so $\deg_{G_0} (j)=4$, which is a contradiction.
Therefore ${\cal N} [G_0]$ contains all the subsets of cardinality 4.
But this is not possible, because there is no graph of order 5 with all the vertices of degree 3.

At this point, using the foregoing four facts, we will prove that there exists a graph $G\in {\cal C}_5$ such that $\DD (G)\le \DD (G_0)$.

We distinguish three cases:
${\cal N} [G_0]$ has exactly one element of size 3;
${\cal N} [G_0]$ has at least two elements $X$ and $Y$ of size 3 with $|X\cap Y|=2$;
and
${\cal N} [G_0]$ has at least two elements $X$ and $Y$ of size 3 with $|X\cap Y|=1$.

First suppose that ${\cal N} [G_0]$ has exactly one element of size 3.
Hence, the remaining elements of  ${\cal N} [G_0]$ have size 4.
We may assume that $\{ 1,2,3 \}\in {\cal N} [G_0]$ and $N_{G_0}[1]=\{1,2,3\}$.
In such a case, $ {\cal N} [G_0] \subseteq \{ \{1,2,3\}, \{ 1,2,4,5\}, \{ 1,3,4,5 \}, \{ 2,3,4,5 \} \} $.
Since $4,5\not \in N_{G_0}[1]$, hence $1\not \in N_{G_0}[4]$ and $1\not \in N_{G_0}[5]$. Consequently, $N_{G_0}[4]=N_{G_0}[5]=\{ 2,3,4,5 \}$.
Therefore $\{ \{ 1,2 \}, \{ 1,3 \} , \{ 2,4 \}, \{ 3,5 \}, \{ 4,5 \} \}\subseteq E(G_0)$.
So, $G_0$ contains a subgraph $G$ that is isomorphic to the cycle $C_5$ and,
by Lemma~\ref{lem.spanningmenor},
$\DD (G)\le \DD (G_0)$.

Next suppose that ${\cal N} [G_0]$ has at least two elements $X$ and $Y$ of size $|X|=|Y|=3$ with $|X\cap Y|=2$.
Without loss of generality we may assume that  $X=\{ 1,2,3 \}$ and that $Y=\{ 1,2,5 \}$.
Since $X,Y\in{\cal N} [G_0]$, and since $\{4,5\}\cap X=\emptyset$ and $\{3,4\}\cap Y= \emptyset$,
we get that $\{ 4,5 \}, \{3,4 \} \notin tr ( {\cal N} [G_0] )= \DD (G_0)$. Therefore $\{ 4,5 \}, \{3,4 \} \notin {\cal H}$ and thus, 
as we have showed before, we conclude that $\{ 4, 5 \}, \{3, 4 \} \in E(G_0)$.
Hence it follows that  $ \{ \{ 1,2,3 \}, \{ 1,2,5 \} \} = \{ N_{G_0}[1], N_{G_0}[2] \} $.
By symmetry, we may assume that  $N_{G_0}[1]=\{ 1,2,3 \}$ and  that $N_{G_0}[2]=\{ 1,2,5 \}$.
In such a case, $\{ \{ 1,2\}, \{1,3\}, \{ 2,5\}, \{4,5 \}, \{ 3, 4 \}  \}\subseteq E(G_0)$.
Hence, $G_0$ contains a subgraph $G$ that is isomorphic to the cycle $C_5$ and,
by Lemma~\ref{lem.spanningmenor},
$\DD (G)\le \DD (G_0)$.

Finally, suppose that ${\cal N} [G_0]$ has at least two elements $X$ and $Y$ of size $|X|=|Y|=3$ with $|X\cap Y|=1$.
Without loss of generality we may assume that
$X=\{ 1,2,3 \}$ and that $Y=\{ 3,4,5 \}$.
Reasoning as in the preceding case, $\{ 4, 5 \}$ and $\{ 1,2 \}$ belong to $E(G_0)$.
If $N_{G_0}[3]=\{ 1,2,3\}$, then  $\{ 3,4,5 \}$ must be either $N_{G_0}[4]$ or $N_{G_0}[5]$, 
obtaining respectively that either $\{3,4 \} \in E(G_0)$ or that $\{3,5 \} \in E(G_0)$. So, if $N_{G_0}[3]=\{ 1,2,3\}$, then we get that  
either $4\in N_{G_0}[3]$ or $5\in N_{G_0}[3]$, a contradiction.
Therefore we conclude that $N_{G_0}[3]\not= \{ 1,2,3\}$  and, by symmetry,  we get that $N_{G_0}[3]\not= \{ 3,4,5 \}$.
Hence, without loss of generality we may assume that $N_{G_0}[1]=\{ 1,2,3\}$ and that $N_{G_0}[4]=\{ 3,4,5 \}$.
At this point recall that the intersection of all the closed neighborhoods in ${\cal N} [G_0]$ 
is empty (because otherwise there would be a vertex $u$ of degree 4).
Set $Z\in {\cal N} [G_0]$ such that $3\notin Z$.
If $|Z|=3$, then either $|X\cap Z|=2$ or $|Y\cap Z|=2$, and we proceed as in the preceding case.
If $|Z|\neq3$, then $Z=\{ 1,2,4,5 \}$, and so $Z$ is either $N_{G_0}[2]$ or $N_{G_0}[5]$.
In any case, $\{ 2,5 \}\in E(G_0)$.
Therefore, $\{ \{1,2 \},\{1,3\},\{ 2,5\},\{3,4\},\{4,5\}  \}\subseteq E(G_0)$.
So, $G_0$ contains a subgraph $G$ that is isomorphic to the cycle $C_5$ and,
by Lemma~\ref{lem.spanningmenor}, $\DD (G)\le \DD (G_0)$.
\end{proof}

\begin{proposition}\label{prop.c5k23minimals}
Let ${\cal H}$ be a domination completion of  ${\cal U}_{3,\Omega}$. Then, 
there exists a graph $G\in  {\cal C}_5 \cup {\cal K}_{2,3}$ such that ${\cal U}_{3,\Omega}\le  \DD (G) \le {\cal H}$.
\end{proposition}

\begin{proof}
Let  ${\cal H}_0=\DD(G_0)$ be a minimal domination completion of ${\cal U}_{3,\Omega}$ such that  ${\cal U}_{3,\Omega}\le  {\cal H}_0 \le  {\cal H}$.
By  Lemmas~\ref{lem.size3} and \ref{lem.size2}, it is enough to show that ${\cal H}_0$ has either an element of size $3$ or all its elements have size 2.
Let us prove it.

First observe that for all $A\in {\cal H}_0$, we have $|A|\le 3$.
Indeed, suppose on the contrary that there exists $A\in {\cal H}_0$ such that $|A|\ge 4$.
If $\{ a,b,c,d \}\subseteq A$, then ${\cal H}_0$ does not contain any subset of $\{ a,b,c \}\in {\cal U}_{3,\Omega}$, 
contradicting that ${\cal U}_{3,\Omega} \leq {\cal H}_0$.

From the above, it only remains to prove that,
if  ${\cal H}_0$ has no elements of size $3$,
then all its elements have size exactly 2.
On the contrary, let us assume that there exists $A\in {\cal H}_0$ such that $|A|=1$. We are going to prove that, in such a case, 
a contradiction is achieved.

Without loss of generality we may assume that $A=\{ 5 \}$. Hence, $\deg_{G_0} (5)=4$ because $A\in {\cal H}=\DD(G_0)$.
Let 
$G_1=G_0-5$ be the graph obtained by deleting the vertex 5 from $G_0$.
It is clear that $G_0= G_1\vee K_{ \{ 5 \} }$. 
Since ${\cal U}_{3,\Omega}\le {\cal H}_0$, hence  ${\cal U}_{3,\Omega}\leqslant \DD (G_1 \vee K_{ \{ 5 \} })$,
and thus, by applying Lemma~\ref{lemacon} it follows that 
$\mathcal{U}_{3,\Omega\setminus \{ 5 \}} = \mathcal{U}_{3,\Omega} [{\Omega\setminus \{ 5 \}}]\le \DD (G_1)$.
Let $\DD (G')$ be a minimal domination completion of $\mathcal{U}_{3,\Omega\setminus \{ 5 \}}$ such that
$\mathcal{U}_{3,\Omega\setminus \{ 5 \}} \le \DD (G')\le \DD (G_1)$.
By using Lemma~\ref{lema.operacions}, it is easy to check that 
$\mathcal{U}_{3,\Omega} \le \DD (G'\vee K_{\{ 5 \} })$ and that $\DD (G'\vee K_{\{ 5 \} })\le \DD (G_1\vee K_{\{ 5 \} })$.
Therefore, $\mathcal{U}_{3,\Omega} \le \DD (G'\vee K_{\{ 5 \} })\le {\cal H}_0$.
So, ${\cal H}_0=\DD (G'\vee K_{\{ 5 \} })$ because ${\cal H}_0$ is a minimal domination completion of $\mathcal{U}_{3,\Omega}$.
By Theorem~\ref{thm.nmenysu} we may assume that $G'$ is isomorphic to $K_{1,3}$ or to $2\, K_2$.
If $G'$ is isomorphic to $K_{1,3}$,
then ${\cal H}_0=  \DD (G'\vee K_{\{ 5 \} })=\{ \{ 5 \} \} \cup \DD (G')$ has an element of size 3, which contradicts our assumption.
If $G'$ is isomorphic to $2\, K_{2}$, then by applying Proposition~\ref{prop.realitzacions}
we get that $\DD (G')=\DD (G'')$ where $G''$ is a path of order 4 obtained by joining a pair of vertices of 
the different connected components in $2 K_2$.
Therefore, ${\cal H}_0=\DD (K_{\{5\}} \vee G')= \{ \{ 5 \} \} \cup \DD (G')=\{ \{ 5 \} \} \cup \DD (G'')=\DD (K_{\{5\}} \vee G'')$. 
But the graph $K_{\{5\}} \vee G''$ contains a spanning subgraph $G'''\in {\cal C}_5$.
So, from Lemma~\ref{lem.spanningmenor} we get that
${\cal U}_{3,\Omega} \le \DD (G''')\le \DD (K_{\{5\}} \vee G'')={\cal H}_0$. 
Therefore, ${\cal H}_0=\DD (G''')$ because ${\cal H}_0$ is a minimal domination completion of ${\cal U}_{3,\Omega}$.
This leads us to a contradiction since all the dominating sets of $G'''$ have size 2.
\end{proof}



\end{document}